\documentclass{amsart}

\usepackage{graphicx}
\usepackage{mathtools}
\usepackage{tikz-cd}
\usepackage{color}
\usepackage{geometry}
\usepackage{amssymb}
\usepackage{adjustbox}

\usepackage{lscape}
\usepackage{mathrsfs}
\usepackage{breqn}
\usepackage{bbm}
\usepackage{wasysym}

\usepackage{hyperref}
\hypersetup{
    colorlinks,
    citecolor=black,
    filecolor=black,
    linkcolor=black,
    urlcolor=black
}

\newtheorem{theorem}{Theorem}[section]
\newtheorem{lemma}[theorem]{Lemma}

\newtheorem{cor}[theorem]{Corollary}
\newtheorem{prop}[theorem]{Proposition}

\theoremstyle{definition}
\newtheorem{definition}[theorem]{Definition}

\theoremstyle{remark}
\newtheorem{remark}[theorem]{Remark}

\numberwithin{equation}{section}

\newcommand{\mA}{{\mathbb A}}

\newcommand{\mC}{{\mathbb C}}

\newcommand{\mQ}{{\mathbb Q}}

\newcommand{\mZ}{{\mathbb Z}}

\newcommand{\cH}{{\mathcal H}}

\newcommand{\cM}{{\mathcal M}}

\newcommand{\fC}{{\mathfrak{C}}}

\newcommand{\sB}{{\mathscr{B}}}

\newcommand{\sD}{{\mathscr{D}}}
\newcommand{\sF}{{\mathscr{F}}}

\newcommand{\sI}{\mathscr{I}}
\newcommand{\sJ}{\mathscr{J}}
\newcommand{\sK}{{\mathscr{K}}}

\newcommand{\sS}{{\mathscr{S}}}

\newcommand{\zinf}{\{0-\infty\}}
\newcommand{\Fil}{\text{\normalfont{Fil}}}

\newcommand{\Gal}{\text{\normalfont{Gal}}}

\newcommand{\Exp}{\text{\normalfont{Exp}}}

\newcommand{\LP}{\text{\normalfont{LP}}}
\newcommand{\D}{\mathbf{D}}
\newcommand{\B}{\mathbf{B}}
\newcommand{\z}{\mathbf{z}}
\newcommand{\N}{\mathbf{N}}
\newcommand{\GL}{\text{\normalfont{GL}}}

\newcommand{\dR}{\text{\normalfont{dR}}}

\newcommand{\cycl}{\text{\normalfont{cycl}}}

\newcommand{\rig}{\text{\normalfont{rig}}}
\newcommand{\unr}{\text{\normalfont{unr}}}
\newcommand{\Res}{\text{\normalfont{Res}}}
\newcommand{\dif}{\text{\normalfont{dif}}}
\newcommand{\crys}{\text{\normalfont{crys}}}
\newcommand{\Proj}{\text{\normalfont{Proj}}}
\newcommand{\Tr}{\text{\normalfont{Tr}}}
\newcommand{\Id}{\text{\normalfont{Id}}}

\newcommand{\HH}{\text{\normalfont{H}}}

\newcommand{\Hom}{\text{\normalfont{Hom}}}
\newcommand{\LA}{\text{\normalfont{LA}}}

\newcommand{\bb}{\text{\normalfont{b}}}

\newcommand{\cores}{\text{\normalfont{cores}}}
\newcommand{\Iw}{\text{\normalfont{Iw}}}
\newcommand{\Kato}{\text{\normalfont{Kato}}}
\newcommand{\Sym}{\text{\normalfont{Sym}}}

\newcommand{\ind}{\text{\normalfont{ind}}}
\newcommand{\Ind}{\text{\normalfont{Ind}}}
\newcommand{\et}{\text{\normalfont{\'et}}}
\newcommand{\sm}{\text{\normalfont{sm}}}
\newcommand{\hw}{\text{\normalfont{hw}}}
\newcommand{\res}{\text{\normalfont{res}}}

\newcommand{\la}{\text{\normalfont{la}}}

\newcommand{\lalg}{\text{\normalfont{lalg}}}

\newcommand{\per}{\text{\normalfont{per}}}

\newcommand{\pool}{\adjustbox{scale=0.4}{$\begin{pmatrix} p & 0 \\
0 & 1
\end{pmatrix}$}=1}

\newcommand{\new}{\text{\normalfont{new}}}

\newcommand{\DD}{\widetilde{\D}(V_f(1))}
\newcommand{\Dd}{\widetilde{\D}^+(V_f(1))\left[\frac{1}{\varphi^r(T)}\text{, }r\geq 0\right]}
\newcommand{\dd}{\widetilde{\D}^+(V_f(1))}
\newcommand{\DDd}{\Dd\bigg/\dd}

\newcommand{\Bb}{\widetilde{\B}^+\left[\frac{1}{\varphi^r(T)}\text{, }r\geq 0\right]}

\newcommand{\pppp}{\adjustbox{scale=0.5}{
  $\begingroup 
   \setlength\arraycolsep{2pt}
   \begin{pmatrix} 
     p & 0 \\
     0 & 1
   \end{pmatrix}
   \endgroup$
   }}

\begin{document}

\title{Completed cohomology and Kato's Euler system for modular forms}

\author{Yiwen Zhou}
\address{Department of Mathematics, University of Chicago.}
\email{yiwen@math.uchicago.edu}

\date{December 8, 2018.}

\begin{abstract} In this paper, we compare two different constructions of $p$-adic $L$-functions for modular forms and their relationship to Galois cohomology: one using Kato's Euler system and the other using Emerton's $p$-adically completed cohomology of modular curves. At a more technical level, we prove the equality of two elements of a local Iwasawa cohomology group, one arising from Kato's Euler system, and the other from the theory of modular symbols and $p$-adic local Langlands correspondence for $\GL_2(\mQ_p)$. We show that this equality holds even in the cases when the construction of $p$-adic $L$-functions is still unknown (i.e. when the modular form $f$ is supercuspidal at $p$). Thus, we are able to give some representation-theoretic descriptions of Kato's Euler system. 
\end{abstract}

\maketitle

\tableofcontents 

\section*{Introduction}

\noindent 

\vskip 0.5em

\noindent Let $f$ be a cuspidal newform of weight $k\geq 2$ and level $\Gamma_0(p^n)\cap \Gamma_1(N)$, $V_f$ be the $p$-adic Galois representation attached to $f$. 
In \cite{K}, Kato constructed for every cuspidal Hecke eigenform $f$ an Euler system $\z_{\Kato}(f)$ in the \emph{global} Iwasawa cohomology of the dual Galois representation $V_f^*$ attached to $f$. The element $\z_{\Kato}(f)$ is the key object for studying both the $p$-adic interpolation properties of classical $L$-functions and the $p$-adic arithmetic information of the modular form $f$. For example, an important property of $\z_{\Kato}(f)$ we will be using in this paper is that the images of $\z_{\Kato}(f)$ under various dual exponential maps compute the special values of the classical $L$-functions of $f$ and its twists by characters of $p$-power conductors. We will describe the precise statement in Theorem \ref{Kato} below.

\vskip 0.5em

Another approach to studying the $p$-adic interpolation properties of classical $L$-functions of modular forms is via modular symbols. In \cite{E} and \cite{E1}, Emerton rephrases this approach by regarding the modular symbol $\{0- \infty\}$ as a functional on $p$-adically complected cohomology of modular curves.   Combining the construction of this functional with the $p$-adic local-global compatibility \cite{E2} and Colmez's theory of $p$-adic local Langlands correspondence \cite{C2}, we may regard $\{0- \infty\}$ as an element $\z_M(f)$ (the letter ``$M$" stands for modular symbols) in the local Iwasawa cohomology of $V_f^*$. 

\vskip 0.5em

The precise definition of $\z_M(f)$ is as follows: according to \cite{CC} and \cite{C2}, there are isomorphisms $$\Exp^*: \text{\quad}\HH^1_{\Iw}(\mathbb{Q}_p, V_f^*)\xrightarrow{\cong}\D(V_f^*)^{\psi=1}$$
and $$\mathfrak{C}: \text{\quad}\left(\Pi(V_f)^{*}\right)^{\pppp=1}\xrightarrow{\cong}\D(V_f^*)^{\psi=1},$$
where $\D(V_f^*)$ is the $(\phi,\Gamma)$-module attached to $V_f^*$, $\Pi(V_f)$ is the $p$-adic Banach space representation attached to $V_f$ by the $p$-adic local Langlands correspondence\footnote{Our notation is slightly different from \cite{C2}: the $\Pi(V)$ in this paper is denoted by $\Pi(V(1))$ in \cite{C2}.}.

Using Emerton's local-global compatibility \cite{E2}, the modular form $f$ gives a (pair of) $\GL_2(\mathbb{Q}_p)$-equivariant embedding (which is actually a $\GL_2(\mQ_p)$-equivariant isomorphism)
\begin{equation}\label{emb}
\Phi_f^{\pm}:\text{\quad}\Pi(V_f)\hookrightarrow \widetilde{H}_c^{1}(K^p)^{\pm,f}
\end{equation}
where $K^p$ is the tame level of $f$, $\widetilde{H}_c^{1}(K^p)$ is the completed cohomology of tame level $K^p$.

We denote the composite $$\Pi(V_f)\xrightarrow{\Phi_f^{\pm}} \widetilde{H}_c^{1}(K^p)^{\pm, f}\xrightarrow{\{0-\infty\}}\mathbb{C}_p$$as an element in $\Pi(V_f)^{*}$ by $\mathcal{M}_f^{\pm}$. In fact, we have (Lemma \ref{lem-key}) $$\mathcal{M}_f^{\pm}\in \left(\Pi(V_f)^{*}\right)^{\adjustbox{scale=0.5}{$\begin{pmatrix} p & 0 \\
0 & 1
\end{pmatrix}$}=1}.$$

We define $\z_M^{\pm}(f):=\left(\Exp^{*}\right)^{-1}\circ \mathfrak{C} (\mathcal{M}_f^{\pm})\in \HH^1_{\Iw}(\mathbb{Q}_p,V_f^*)$, and $\z_M(f):=\z_M^+(f)-\z_M^-(f)\in \HH^1_{\Iw}(\mathbb{Q}_p,V_f^*)$.

\vskip 1em

The cohomology classes $\z_{\Kato}(f)$ and $\z_M(f)$ arise from very different perspectives of the $p$-adic arithmetic of the modular form $f$. On the other hand, both of them lie in the Iwasawa cohomology group attached to $f$, and in fact both encode special $L$-values of $f$ and its twists. Thus it is natural to ask what the relationship between these two elements is. Our main result answers this question:

\vskip 1em

\begin{theorem}\label{main}
If $V_f|_{G_{\mathbb{Q}_p}}$ is absolutely irreducible, then $\z_M(f)=\z_{\Kato}(f)$ as elements in $\HH^1_{\Iw}(\mathbb{Q}_p,V_f^*)$.
\end{theorem}

\vskip 0.5em




From the perspective of the special $L$-values they encode, $\z_{\Kato}(f)$ comes from the Rankin-Selberg method, and $\z_M(f)$ comes from the Mellin transforms, so the above theorem compares these two different integral formulas for $L$-functions of modular forms. It would be interesting to have a direct comparison, but our approach is to use the equality of special values under dual exponential maps (computed either way for $\z_{\Kato}(f)$ and $\z_M(f)$) as the basic input.

\vskip 0.5em

The strategy of proving Theorem \ref{main} is to show that the images of both  $\z_M(f)$ and $\z_{\Kato}(f)$ under various dual exponential maps are the same, and then use Proposition II.3.1 of \cite{Ber1} to conclude that $\z_M(f)=\z_{\Kato}(f)$ as elements in $\HH^1_{\Iw}(\mathbb{Q}_p,V_f^*)$.

\vskip 0.5em

\begin{theorem}\label{Z}
If $V_f|_{G_{\mathbb{Q}_p}}$ is absolutely irreducible, then the element $\z_{M}(f)\in\HH^1_{\Iw}(\mathbb{Q}_p,V_f^*)$ satisfies the following property:

For any $0\leq j\leq k-2$, any finite order character $\phi$ of $\mathbb{Z}_p^{*}\cong\Gamma_{\mathbb{Q}_p}$ with conductor $p^n$, $$\exp^*\left(\int_{\mathbb{Z}_p^{*}}\phi(x)x^{-j}\cdot\z_{M}(f)\right)=\frac{1}{\tau(\phi)}\cdot\frac{\widetilde{\Lambda}_{(p)}(f,\phi, j+1)}{j!}\cdot \overline{f}_{\phi}\cdot t^j,$$where
$$\widetilde{\Lambda}_{(p)}(f,\phi, j+1)=\frac{\Gamma(j+1)}{(2\pi i)^{j+1}}\cdot\frac{L_{(p)}(f,\phi, j+1)}{\Omega_f^{\pm}}\in\mathbb{Q}(f,\mu_{p^n}),$$
$L_{(p)}(f,\phi,j+1)$ is obained by removing the Euler factor at $p$ from the classical $L$-function of $f$ twisted by $\phi$, $\tau(\phi)$ is the Gauss sum of the character $\phi$, and $$\overline{f}_{\phi}\cdot t^j:=\left(\sum_{a}\phi(a)\overline{f}\otimes\zeta_{p^n}^a\right)\cdot t^j=\overline{f}\cdot e_{\phi}\cdot t^j\in \Fil^0\D_{\dR}(V_f^*(\phi)(-j)).$$  
\end{theorem}

Two key ingredients used in proving Theorem \ref{Z} are the theory of $p$-adic Kirillov models of locally algebraic representations of $\GL_2(\mathbb{Q}_p)$ introduced in \cite{C2} (Section VI.2.5), and an ``explicit reciprocity law" introduced in \cite{C2} (Proposition VI.3.4), \cite{D} (Th\'{e}or\`{e}me 8.3.1) and \cite{EPW} (Theorem 5.4.3).

\vskip 1em

On the other hand, Kato in \cite{K} showed that $\z_{\Kato}(f)$ has the following interpolation property:
  
\begin{theorem}[Kato, \cite{K}, Theorem 12.5]\label{Kato}
The element $\z_{\Kato}(f)\in\HH^1_{\Iw}(\mathbb{Q}_p,V_f^*)$ satisfies the following property:

For any $0\leq j\leq k-2$, any finite order character $\phi$ of $\mathbb{Z}_p^{*}\cong\Gamma_{\mathbb{Q}_p}$ with conductor $p^n$, $$\exp^*\left(\int_{\mathbb{Z}_p^{*}}\phi(x)x^{-j}\cdot\z_{\Kato}(f)\right)=\frac{1}{\tau(\phi)}\cdot\frac{\widetilde{\Lambda}_{(p)}(f,\phi, j+1)}{j!}\cdot \overline{f}_{\phi}\cdot t^j.$$
\end{theorem}

\vskip 1em

Comparing the formulas in Theorem \ref{Z} and Theorem \ref{Kato}, and using Proposition II.3.1 of \cite{Ber1}, we obtain Theorem \ref{main}.



\vskip 1em

\noindent\textbf{Organization of the paper.} We compute in Section 2 the evaluations of the modular symbol $\{0-\infty\}$ on locally algebraic vectors of $\Pi(V_f)$ under the embedding $\Phi_f^{\pm}: \Pi(V_f)\hookrightarrow \hat{H}_c^{1}(K^p)^{\pm,f}$ given by the modular form $f$. 

Section 3 - 7 compute the images of $\z_M(f)$ under various dual exponential maps, with the assumption that $V_f|_{G_{\mQ_p}}$ is absolutely irreducible. These sections are divided into two parts: Section 3 - 5 deal with the supercuspidal case (meaning that the smooth representation of $\GL_2(\mQ_p)$ attached to $f$ is supercuspidal), and Section 6 \& 7 deal with principal series case (meaning that the smooth representation of $\GL_2(\mQ_p)$ attached to $f$ is a principal series). 

 In Section 3, we introduce the theory of $p$-adic Kirillov models for locally algebraic representations of $\GL_2(\mQ_p)$, and use the explicit formulas of the $p$-adic Kirillov models to compute the images of locally algebraic vectors of $\Pi\left(V_f\right)$ under the maps $\iota^-_m$. In section 4, we describe an ``explicit reciprocity law", and use it to compute the image of $\z_M^{\pm}(f)$ under the maps $\iota^-_m$. In section 5, we present and give a proof of the formulas for  the images of $\z_M(f)$ under various dual exponential maps in the supercuspidal case. In section 6, we describe explicitly the $p$-adic Local Langlands correspondence, and then use results from \cite{E} to describe explicitly the image of $\z_M^{\pm}(f)$ in $\B^+_{\rig,\mQ_p}\otimes_{\mQ_p}\D_{\crys}(V_f^*)$. In section 7, we deduce the formulas for  the images of $\z_M(f)$ under various dual exponential maps in the principal series case.

In Section 8, we conclude $\z_M(f)=\z_{\Kato}(f)$ in both the supercuspidal and the principal series cases, under the assumption that $V_f|_{G_{\mQ_p}}$ is absolutely irreducible.

\vskip 1em

\subsection*{Acknowledgement} I'd like to thank my advisor, Matthew Emerton, for suggesting me working on this project and sharing with me his profond insight and ideas. I'd like to thank Robert Pollack for pointing out one key step of reaching the final result of this paper, and for kindly sharing his drafts and notes with me. I'd like to thank Lue Pan, who has provided me endless help and encouragement on this project. I'd like to thank Gong Chen, Pierre Colmez, Yiwen Ding, Yongquan Hu, Joaqu\'{i}n Rodrigues Jacinto, Kazuya Kato, Ruochuan Liu, Xin Wan, Haining Wang, Jun Wang, Shanwen Wang for very helpful discussions.

\vskip 3em

\section{Notations and conventions}\label{notation}

\noindent We denote $\Gamma$ the Galois group $\Gal\left(\mQ_p(\zeta_{p^{\infty}})/\mQ_p\right)$, and $H$ the Galois group $\Gal\left(\overline{\mQ}_p/\mQ_p(\zeta_{p^{\infty}})\right)$. We identify $\Gamma$ with $\mZ_p^*$ via the cyclotomic character $\cycl: \Gamma\xrightarrow{\cong} \mZ_p^*$.

$P\left(\mQ_p\right)$ is the subgroup of $\GL_2\left(\mQ_p\right)$ consisting of matrices of the form $\begin{pmatrix} \mQ_p^* & \mQ_p \\
0 & 1
\end{pmatrix}$.

\vskip 1em

Throughout this paper, $f$ denotes a (normalized) classical cuspidal newform of weight $k\geq 2$ and level $\Gamma_0(p^n)\cap \Gamma_1(N)$. $V_f$ is the cohomological Galois representation of $G_{\mQ}$ attached to $f$. Thus the restriction $V_f|_{G_{\mQ_p}}$ has Hodge-Tate weight $0$ and $1-k$.\footnote{Here, the convention is that the cyclotomic character has Hodge-Tate weight $1$.} 

Let $F$ be a finite extension of $\mQ$ which is sufficiently large, that is, containing all the Fourier coefficients of $f$ and values of $\chi$ when a finite order character $\chi$ is chosen. We choose $\lambda$ a place of $F$ lying over $p$, and denote $L:=F_{\lambda}$ the localization of $F$ at $\lambda$. For any integer $n\geq 0 $, we write $L_n:=L\otimes_{\mQ_p}\mQ_p(\zeta_{p^n})$. Here, $\left\{\zeta_{p^n}\right\}_{n\geq 1}$ is chosen to be a compatible system of $p$-power roots of unity. We also write $\mQ_{p,n}:=\mQ_p(\zeta_{p^n})$

We denote $\pi_p^{\sm}(f)$ the $p$-adic smooth representation of $\GL_2(\mQ_p)$ attached to $f$, and $\Pi(V_f)$ the $p$-adic Banach space representation of $\GL_2(\mQ_p)$ attached to $f$. Our notation is slightly different from \cite{C2}: the $\Pi(V)$ in this paper is denoted by $\Pi(V(1))$ in \cite{C2}.

We will only work with the cases when $V_f|_{G_{\mQ_p}}$ is absolutely irreducible. This happens precisely when $\pi_p^{\sm}(f)$ is supercuspidal, or when $\pi_p^{\sm}(f)$ is some twist of an unramified principal series.

The modular form $f$ is a holomorphic section of some line bundle on the modular curve, therefore can be viewed naturally as an element in $\Fil^0\D_{\dR}(V_f(k-1))$. Similarly, the complex conjugate of $f$, which we denote by $\overline{f}$, can be viewed naturally as an element in $\Fil^0\D_{\dR}(V_{\overline{f}}(k-1))=\Fil^0\D_{\dR}(V^*_f)$.


There is a natural pairing $$[\text{ , }]_{\dR}\text{:\quad}\D_{\dR}\left(V^*_f\right)\times \D_{\dR}\left(V_f(1)\right)\rightarrow\D_{\dR}(L(1))\cong\frac{1}{t}L,$$ under which $\Fil^0\D_{\dR}\left(V^*_f\right)$ and $\Fil^0\D_{\dR}\left(V_f(1)\right)$ are orthogonal complement of each other. We define $\overline{f}^*$ the unique element in $\D_{\dR}\left(V_f(1)\right)/\Fil^0\D_{\dR}\left(V_f(1)\right)$ such that $[\overline{f},\overline{f}^*]_{\dR}=\frac{1}{t}$.

\vskip 3em

\section{Special values of \texorpdfstring{$L$}{L}-function in terms of modular symbols}\label{sec-1}

\noindent In this section, we compute the modular symbol $\{0-\infty\}$ evaluated at locally algebraic vectors of $\Pi(V_f)$ under the embedding \eqref{emb}.  

\vskip 1em

Let $f$ be a cuspidal newform of weight $k\geq 2$ and level $\Gamma_0(p^n)\cap \Gamma_1(N)$. We denote $K^p$ the tame level of $f$, i.e. $$K^p:=\left\{\left.\begin{pmatrix}
a & b \\ c & d
\end{pmatrix}\in\GL_2(\widehat{\mZ}^p)\right| c\equiv 0 \text{ mod }N\text{, }d\equiv 1 \text{ mod }N
\right\},$$ and $Y(Np^n)$ the modular curve defined over $\mathbb{Q}$ whose $\mathbb{C}$ points are given by $$Y(Np^n)(\mathbb{C})=\GL_2(\mathbb{Q}) \Big\backslash \GL_2(\mathbb{A}_{\mathbb{Q}})\Big/ \Gamma(p^n)\cdot K^p \cdot\mathbb{C}^{\times}.$$


Let $Y(Np^n)_{\mathbb{Q}(\zeta_{Np^n})}$ denote the base change to $\mathbb{Q}(\zeta_{Np^n})$ of $Y(Np^n)$, and $Y^{\circ}(Np^n)$ denote the connected component of $Y(Np^n)_{\mathbb{Q}(\zeta_{Np^n})}$. We also write $X^{\circ}(Np^n)$ (resp. $X(Np^n)$) the compactification of $Y^{\circ}(Np^n)$ (resp. $Y(Np^n)$).


Fix an embedding $\iota:\overline{\mathbb{Q}}\hookrightarrow \mathbb{C}$. There is an action of complex conjugation $\tau$ on $Y(Np^n)_{\mathbb{Q}(\zeta_{Np^n})}$. The action of $\tau$ on $\pi_0\left(Y(Np^n)_{\mathbb{Q}(\zeta_{Np^n})}\right)\subset \left(\mathbb{Z}/(Np^n)\right)^{\times}$ coincides with the one induced from multiplication by $-1$ on $\left(\mathbb{Z}/(Np^n)\right)^{\times}$. 

We let $F$ be a finite extension of $\mQ$ that contains all Fourier coefficients of $f$ and values of $\chi$ when a finite order character $\chi$ is chosen in the future. Choose $\lambda$ a place of $F$ lying over $p$, and denote $L:=F_{\lambda}$ the localization of $F$ at $\lambda$. The completed cohomology of the modular curve with tame level $K^p$ and coefficient $L$ is defined as $$\widetilde{H}_c^{1}(K^p):=L\otimes_{\mathcal{O}_L}\lim\limits_{\substack{\longleftarrow \\ m}} \lim\limits_{\substack{\longrightarrow \\ n}} H^1_{\et,c}\left(Y(Np^n)_{\overline{\mathbb{Q}}},\text{ }\mathbb{Z}/p^m\right). $$ This is a $p$-adic Banach space equiped with a continuous action of $\GL_2(\mQ_p)\times \mathcal{H}^p\times \{\Id, \tau\}$.

\vskip 1em

The modular symbol $\zinf\in H_1(X^{\circ}(Np^n)(\mathbb{C}),\text{cusps};\mathbb{Z})\subset H_1(X(Np^n)(\mathbb{C}),\text{cusps};\mathbb{Z})$ for all $n$, and are compatible with respect to the map $ H_1(X(Np^{n+1})(\mathbb{C}),\text{cusps};\mathbb{Z})\rightarrow H_1(X(Np^n)(\mathbb{C}),\text{cusps};\mathbb{Z})$ induced from the natural projection $X(Np^{n+1})\rightarrow X(Np^{n})$. Note that the homology theory we are using here is the singular homology.

We then have for every positive integer $m$, $$\zinf\in \lim\limits_{\substack{\longleftarrow \\ n}} H_1(X(Np^n)(\mathbb{C}),\text{cusps};\mathbb{Z}/p^m)=\Hom\left(\lim\limits_{\substack{\longrightarrow \\ n}} H^1_c\left(Y(Np^n)(\mathbb{C}),\mathbb{Z}/p^m\right),\mathbb{Z}/p^m\right).$$

Under the identification $$H^1\left(Y(Np^n)(\mathbb{C}),\mathbb{Z}/p^m\right)\cong H^1_{\et,c}\left(Y(Np^n)_{\overline{\mathbb{Q}}},\mathbb{Z}/p^m\right),$$

we can view  \begin{equation}
    \zinf\in \Hom\left(\lim\limits_{\substack{\longrightarrow \\ n}}H^1_{\et,c}\left(Y(Np^n)_{\overline{\mathbb{Q}}},\mathbb{Z}/p^m\right),\mathbb{Z}/p^m\right)
\end{equation} for every $m$, hence \begin{equation}
   \zinf\in \left(\widetilde{H}_c^{1}(K^p)\right)'_{\bb}
\end{equation}
where the right hand side is the bounded dual of the $p$-adic Banach space $\widetilde{H}_c^{1}(K^p)$.

\vskip 1em

We denote $W_{k-2}$ the contragredient to the $(k-2)$-nd symmetric power of the standard representation of $\GL_2$. Since $f\in H^0\left(Y(Np^n),\omega^{k-2}\otimes\Omega^1\right)$ is a holomorphic cuspidal newform of weight $k\geq 2$, we have a period map: $$\per: H^0\left(Y(Np^n),\omega^{k-2}\otimes\Omega^1\right)\rightarrow H^1_c\left(Y(Np^n)(\mathbb{C}),\mathcal{V}_{\check{W}_{k-2}(\mathbb{\overline{Q}})}\right)\otimes_{\mathbb{\overline{Q}}}\mathbb{C}.$$
If we denote by $i^*$ the restriction $$i^*: H^1_c\left(Y(Np^n)(\mathbb{C}),\mathcal{V}_{\check{W}_{k-2}(\mathbb{\overline{Q}})}\right)\otimes_{\mathbb{\overline{Q}}}\mathbb{C}\rightarrow H^1_c\left(Y^{\circ}(Np^n)(\mathbb{C}),\mathcal{V}_{\check{W}_{k-2}(\mathbb{\overline{Q}})}\right)\otimes_{\mathbb{\overline{Q}}}\mathbb{C} =H^1_c\left(\Gamma,\check{W}_{k-2}(\mathbb{C})\right),$$
then we have \begin{equation}\label{2.9}
    i^* \per(f)=\left\{\gamma\in\Gamma\mapsto \sum_{i+j=k-2}\int_{\gamma}f(z)z^idz\cdot e_1^ie_2^j\in\check{W}_{k-2}(\mathbb{C})\right\}.
\end{equation}
Here, we consider $\mathbb{Q}^2$ having the standard basis $e_1$ and $e_2$, and the action of any matrix $\begin{pmatrix} a & b \\
c & d\end{pmatrix}\in\GL_2\left(\mQ\right)$ on $\mathbb{Q}^2$ is given by formulas $\begin{pmatrix} a & b \\
c & d\end{pmatrix}e_1=ae_1+ce_2$ and $\begin{pmatrix} a & b \\
c & d\end{pmatrix}e_2=be_1+de_2$. The basis of $\check{W}_{k-2}(\mQ)=\Sym^{k-2}\mathbb{Q}^2$ is chosen as $e_1^ie_2^j$ for any $i+j=k-2$. 

\begin{remark}\label{2.2} 
If $g=\begin{pmatrix}a & b\\ c & d
    \end{pmatrix}\in\GL_2(\mathbb{R})^+$ normalizes $\Gamma$, then there is an action of $g$ on the group $H^1_c\left(Y^{\circ}(Np^n)(\mathbb{C}),\mathcal{V}_{\check{W}_{k-2}(\mathbb{C})}\right)\cong H^1_c\left(\Gamma, \check{W}_{k-2}(\mathbb{C})\right)$ given by formula: $(g.c)(\gamma)=g.\left(c(g^{-1}\gamma g)\right)$.
    
In particular, we have: \begin{equation}\label{2.11}
    \left(g.\left(i^* \per(f)\right)\right)(\zinf) =\sum_{i+j=k-2}\int_{-\frac{d}{c}}^{-\frac{b}{a}} f(z)(cz+d)^{k-2}\left(\frac{az+b}{cz+d}\right)^idz\cdot e_1^ie_2^{j}.
\end{equation}
\end{remark}

\vskip 1em

The complex conjugation $\tau$ acting on $Y(Np^n)(\mathbb{C})$ induces an action of complex conjugation on $H^1_c\left(Y(Np^n)(\mathbb{C}),\mathcal{V}_{\check{W}_{k-2}(\mathbb{\overline{Q}})}\right)$, which we still denote by $\tau$. We have \begin{equation}
    i^* \tau\left(\per(f)\right)=\left\{\gamma\in\Gamma\mapsto \sum_{i+j=k-2}\int_{\tau(\gamma)} f(z)(-z)^idz\cdot e_1^ie_2^j .\right\}\end{equation}
Here, we are identifying $Y^{\circ}(Np^n)(\mC)$ as a quotient of the complex upper half plane, and the action of $\tau$ on the upper half plane is given by reflecting along the imaginary axis.  

We write $H^1_c\left(Y(Np^n)(\mathbb{C}),\mathcal{V}_{\check{W}_{k-2}(\mathbb{\overline{Q}})}\right)^{\pm}$ the $\pm$ eigenspace of the complex conjugation. We then have the $\pm$ period maps: \begin{equation}
    \per^{\pm}: H^0\left(Y(Np^n),\omega^{k-2}\otimes\Omega^1\right)^f\rightarrow H^1_c\left(Y(Np^n)(\mathbb{C}),\mathcal{V}_{\check{W}_{k-2}(\mathbb{\overline{Q}})}\right)^{\pm,f}\otimes_{\mathbb{\overline{Q}}}\mathbb{C},
\end{equation}
one has for any $0\leq j\leq k-2$, \begin{equation}
    i^* \per^{\pm}(f)=\begin{cases}
\left\{\gamma\in\Gamma\mapsto\sum_{i+j=k-2}\frac{\int_\gamma f(z)z^jdz+\int_{\tau(\gamma)} f(z)z^jdz}{2}\cdot e_1^ie_2^j \right\},&\text{  if  } (-1)^j=\pm 1;\\
\left\{\gamma\in\Gamma\mapsto \sum_{i+j=k-2}\frac{\int_\gamma f(z)z^jdz-\int_{\tau(\gamma)} f(z)z^jdz}{2}\cdot e_1^ie_2^j  \right\},&\text{  if  } (-1)^j=\mp 1.
\end{cases}
\end{equation}

\vskip 1em

It is well known (see for example, \cite{MTT}, page 11) that there are complex numbers $\Omega_f^{\pm}\in\mathbb{C}$ such that for any $0\leq j\leq k-2$, \begin{equation}
    \frac{\int_{\gamma} f(z)z^jdz\pm \int_{\tau(\gamma)} f(z)z^jdz}{2\Omega_f^{\pm}}\in F,
\end{equation}
where the signs in the numerator and the denominator are the same if $ (-1)^j=1$, different if otherwise $ (-1)^j=-1$. Therefore, we can define: 
\begin{equation}
    \per^{\pm}_F: H^0\left(Y(Np^n),\omega^{k-2}\otimes\Omega^1\right)^f\rightarrow H^1_c\left(Y(Np^n)(\mathbb{C}),\mathcal{V}_{\check{W}_{k-2}(F)}\right)^{\pm,f},
\end{equation}
using the formula:
\begin{equation}\label{2.14}
    i^* \per_F^{\pm}(f)=\begin{cases}
\left\{\gamma\in\Gamma\mapsto\sum_{i+j=k-2}\frac{\int_\gamma f(z)z^jdz+\int_{\tau(\gamma)} f(z)z^jdz}{2\Omega_f^{\pm}}\cdot e_1^ie_2^j\right\},&\text{ if  }(-1)^j=\pm 1\\
\left\{\gamma\in\Gamma\mapsto\sum_{i+j=k-2}\frac{\int_\gamma f(z)z^jdz-\int_{\tau(\gamma)} f(z)z^jdz}{2\Omega_f^{\pm}}\cdot e_1^ie_2^j\right\},&\text{ if  }(-1)^j=\mp 1
\end{cases}
\end{equation}
where $0\leq j\leq k-2$.

\begin{remark}
We then have: \begin{align}
    \per(f)&=\Omega_f^+\cdot \per^+_F(f)+\Omega_f^-\cdot\per^-_F(f);\\
    \tau(\per(f))&=\Omega_f^+\cdot \per^+_F(f)-\Omega_f^-\cdot\per^-_F(f).
\end{align}
\end{remark}

\vskip 1em

Recall that $\lambda$ is a place of $F$ lying above $p$, and $L=F_{\lambda}$. Let $\per^{\pm}_{L}$ denote the composite of isomorphism $$H^1_c\left(Y(Np^n)(\mathbb{C}),\mathcal{V}_{\check{W}_{k-2}(F)}\right)^{\pm}\otimes_F L\cong H^1_{\et,c}\left(Y(Np^n)_{\overline{\mathbb{Q}}},\mathcal{V}_{\check{W}_{k-2}(L)}\right)^{\pm}$$ with $\per^{\pm}_{F}$.

\vskip 1em

Let $V_f$ be the cohomological Galois representation of $G_{\mQ}$ attached to $f$ with coefficients in $L$. Thus the restriction $V_f|_{G_{\mQ_p}}$ has Hodge-Tate weights $0$ and $1-k$. Let $\Pi\left(V_f\right)$ (with convention same as in \cite{E2}) be the $p$-adic Banach space representation of $\GL_2\left(\mQ_p\right)$ attached to $V_f|_{G_{\mQ_p}}$ via the $p$-adic local Langlands correspondence.

\vskip 1em

As is mentioned in the introduction, according to Emerton's local-global compatibility \cite{E2}, the newform $f$ gives a (pair of) $\GL_2(\mathbb{Q}_p)$-equivariant embedding, which are in fact isomorphisms: $$\Phi_f^{\pm}: \text{\quad}\Pi(V_f)\hookrightarrow \widetilde{H}_c^{1}(K^p)^{\pm,f}\subset \widetilde{H}_c^{1}(K^p)$$
by choosing a new vector at every place away from $p$.

We denote the composite $$\Pi(V_f)\xrightarrow{\Phi_f^{\pm}} \widetilde{H}_c^{1}(K^p)^{\pm, f}\xrightarrow{\{0-\infty\}}\mathbb{C}_p$$as an element in $\Pi(V_f)^{*}$ by $\mathcal{M}^{\pm}_f$. 

Notice that if $g\in\GL_2\left(\mQ\right)$, then there is an action of $g$ on $\widetilde{H}_c^{1}(K^p)$ as a diagonal element in $\GL_2\left(\mQ_p\right)\times\cH^p\times \{\Id, \tau\}$ (induced from the action of $\GL_2\left(\mA_{\mQ}\right)$ on $\widetilde{H}_c^{1}:=\lim\limits_{\substack{\longrightarrow \\ K^p}}\widetilde{H}_c^{1}(K^p)$), where the action of $g$ through the last factor is $\Id$ if $\det(g)$ is positive, and $\tau$ if $\det(g)$ is negative. We have the following lemma:

\begin{lemma}\label{lem-key}
The actions of matrices of the form $\begin{pmatrix} p^{\mZ} & \mZ \\0 & 1\end{pmatrix}$ on  $\widetilde{H}_c^{1}(K^p)$ as a diagonal element in $\GL_2\left(\mQ_p\right)\times\cH^p\times \{\Id, \tau\}$ and as an element in $\GL_2\left(\mQ_p\right)$ coincide. In particular, we have $\mathcal{M}_f^{\pm}\in \left(\Pi(V_f)^{*}\right)^{\adjustbox{scale=0.5}{$\begin{pmatrix} p & 0 \\
0 & 1
\end{pmatrix}$}=1}$.
\end{lemma}
 
\begin{proof}
The first assertion follows from the fact that  $\begin{pmatrix} p^{\mZ} & \mZ \\0 & 1\end{pmatrix}$  belongs to $\begin{pmatrix} * & * \\0 & 1\end{pmatrix}\subset \GL_2\left(\mZ_l\right)$ for every $l\neq p$ and has positive determinant.

The second assertion follows from the fact that the modular symbol $\{0-\infty\}$ is fixed by $\begin{pmatrix} p & 0 \\0 & 1\end{pmatrix}$ as an element in $\GL_2\left(\mQ_p\right)\times\cH^p\times \{\Id, \tau\}$.
\end{proof}

\vskip 1em

We can then define $$\z_M^{\pm}(f):=\left(\Exp^{*}\right)^{-1}\circ \fC (\mathcal{M}_f^{\pm})\in \HH^1_{\Iw}(\mathbb{Q}_p,V_f^*),$$where the isomorphisms $$\Exp^*:\text{\quad} \HH^1_{\Iw}(\mathbb{Q}_p, V_f^*)\xrightarrow{\cong}\D(V_f^*)^{\psi=1}$$
and $$\mathfrak{C}: \text{\quad}\left(\Pi(V_f)^{*}\right)^{\pppp=1}\xrightarrow{\cong}\D(V_f^*)^{\psi=1}$$
are as described in \cite{CC} and \cite{C2}. Here, $\D(V_f^*)$ is the $(\phi,\Gamma)$-module attached to $V_f^*$ as defined in \cite{C2}.

There is an isomorphism: \begin{equation}
    W_{k-2}(L)\otimes_{L}\lim\limits_{\substack{\longrightarrow \\ n}}H^1_{\et,c}\left(Y(Np^n),\mathcal{V}_{\check{W}_{k-2}(L)}\right)\xrightarrow{\cong}\widetilde{H}_c^{1}(K^p)_{W_{k-2}-\lalg}.
\end{equation}

This isomorphicm is equivariant for the complex conjugation $\tau$, where the action of $\tau$ on the left is given by $\Id\otimes\tau$.

We still use $\per^{\pm}_{L}$ to denote the following composition: \begin{equation}
    H^0\left(Y(Np^n),\omega^{k-2}\otimes\Omega^1\right)\xrightarrow{\per^{\pm}_{L}} H^1_{\et,c}\left(Y(Np^n)_{\overline{\mathbb{Q}}},\mathcal{V}_{\check{W}_{k-2}(L)}\right)^{\pm}\rightarrow\left( \lim\limits_{\substack{\longrightarrow \\ n}}H^1_{\et,c}\left(Y(Np^n),\mathcal{V}_{\check{W}_{k-2}(L)}\right)\right)^{\pm}.
\end{equation}

Classical Eichler-Shimura theory tells us that for any choice of sign $\pm$, there is a unique $\GL_2(\mathbb{Q}_p)$-equivariant embedding $$\pi_p^{\sm}(f)\hookrightarrow \left( \lim\limits_{\substack{\longrightarrow \\ n}}H^1_{\et,c}\left(Y(Np^n),\mathcal{V}_{\check{W}_{k-2}(L)}\right)\right)^{\pm,f}$$ where $\pi_p^{\sm}(f)$ is the $p$-adic smooth representation of $\GL_2(\mQ_p)$ attached to $f$ (or equivalently, attached to Weil-Deligne representation associated to $V_f|_{G_{\mQ_p}}$ via the local Langlands correspodence), under which the image of $v_{\new}\in\pi_p^{\sm}(f)$ equals $\per^{\pm}_{L}(f)$.

Thus we have the following commutative diagram: 

\begin{center}
\begin{tikzcd}
  W_{k-2}(L)\otimes_{L}\left(\lim\limits_{\substack{\longrightarrow \\ n}}H^1_{\et,c}\left(Y(Np^n),\mathcal{V}_{\check{W}_{k-2}(L)}\right)\right)^{\pm, f}\ar[r,"\cong"] &  \left(\widetilde{H}_c^{1}(K^p)_{W_{k-2}-\lalg}\right)^{\pm, f}\ar[r,"\zinf"]&\mathbb{C}_p \\
   W_{k-2}(L)\otimes_L\pi_p\ar[r,hook,"\cong"]\ar[u,hook,"\Id\otimes\text{Eichler-Shimura}"]\ar[urr,bend right=40, "\mathcal{M}_f^{\pm}"'] & \Pi(V_f)_{W_{k-2}-\lalg} \ar[u,"\Phi_f^{\pm}"']
\end{tikzcd}
\end{center}

Write the weight vectors of $W_{k-2}$ to be $v_0=v_{\hw}=(e_2^{k-2})^*, \text{ } v_1=(e_1e_2^{k-3})^*, \text{ } \cdots, \text{ } v_{k-3}=(e_1^{k-3}e_2)^*, \text{ } v_{k-2}=v_{\text{lw}}=(e_1^{k-2})^*$.
we then have the following lemma: 

\begin{lemma} For any $0\leq j\leq k-2$,
\begin{equation}
    \mathcal{M}_f^{\pm}\left(v_j\otimes v_{\new}\right)= \left\langle \zinf, v_j\otimes \per_{L}^{\pm}(f)\right\rangle_{(e_1^{j}e_2^{k-2-j})}=\begin{cases}
    \frac{\int_{i\infty}^0 f(z)z^{j}dz}{\Omega_f^{\pm}}\text{ },&\text{\normalfont{if} } (-1)^j=\pm 1;\\
    0\text{ },&\text{\normalfont{if} } (-1)^j=\mp 1.
    \end{cases}
\end{equation}
\begin{flushright}
$\Box$
\end{flushright}
\end{lemma}

\vskip 2em

According to \cite{MTT}, for any finite order chraracter $\chi$ of conductor $p^n$ and any $0\leq j\leq k-2$, we have the formula: \begin{align*}
    L(f,\overline{\chi},j+1)&=\frac{(-2\pi i )^j}{j!}\cdot p^{n(\frac{k}{2}-j-2)}\cdot\tau(\overline{\chi})\cdot \sum_{a\text{ mod }p^n}\chi(a)\cdot 2\pi i \cdot \int_{i \infty}^0\left(f\bigg|^k_{\adjustbox{scale=0.6}{$\begin{pmatrix}1 & -a\\ 0 & p^n
    \end{pmatrix}$}}\right)(z)\cdot z^j dz\\
    &=\frac{(-2\pi i )^j}{j!}\cdot p^{n(-j-1)}\cdot\tau(\overline{\chi})\cdot \sum_{a\text{ mod }p^n}\chi(a)\cdot 2\pi i \cdot \int_{i \infty}^{-\frac{a}{p^n}} f(z)\cdot (p^nz+a)^j dz. 
\end{align*}

\vskip 1em

Notice that if we denote $\{-\frac{a}{p^n}-\infty\}$ by $\gamma$, then by equation \eqref{2.14}, we have: \begin{align*}
    &\int_{i \infty}^{-\frac{a}{p^n}} f(z)\cdot (p^n z+a)^j dz+\int_{i \infty}^{+\frac{a}{p^n}} f(z)\cdot (p^n z-a)^j dz\\
=\text{ }&\int_{\gamma} f(z)\cdot (p^n z+a)^j dz+\int_{\tau(\gamma)} f(z)\cdot (p^n z-a)^j dz\\
=\text{ }&
2\Omega_f^{\pm}\cdot \left\langle \left\{-\frac{a}{p^n}-\infty\right\}, \left(\sum_{t=0}^jC_j^tp^{nt}a^{j-t}v_t\right)\otimes \per^{\pm}_F(f) \right\rangle,&\text{ if  } (-1)^j=\pm 1,
\end{align*}
and similarly, 
\begin{align*}
    &\int_{i \infty}^{-\frac{a}{p^n}} f(z)\cdot (p^nz+a)^j dz-\int_{i \infty}^{+\frac{a}{p^n}} f(z)\cdot (p^nz-a)^j dz\\
=\text{ }&\int_{\gamma} f(z)\cdot (p^nz+a)^j dz-\int_{\tau(\gamma)} f(z)\cdot (p^nz-a)^j dz\\
=\text{ }&
2\Omega_f^{\pm}\cdot \left\langle \left\{-\frac{a}{p^n}-\infty\right\}, \left(\sum_{t=0}^jC_j^tp^{nt}a^{j-t}v_t\right)\otimes \per^{\pm}_F(f) \right\rangle,&\text{ if  }  (-1)^j=\mp 1.
\end{align*}

Therefore, 

\begin{align*}
    &L(f,\overline{\chi},j+1)=\frac{(-2\pi i )^j}{j!}\cdot p^{n(-j-1)}\cdot\tau(\overline{\chi})\cdot \sum_{a\text{ mod }p^n}\chi(a)\cdot 2\pi i \cdot \int_{i \infty}^{-\frac{a}{p^n}} f(z)\cdot (p^nz+a)^j dz\\
    =&\frac{(-2\pi i )^j}{j!}\cdot p^{n(-j-1)}\cdot\tau(\overline{\chi})\cdot \frac{1}{2}\sum_{a\text{ mod }p^n}\chi(a)\cdot 2\pi i \cdot \left(\int_{i \infty}^{-\frac{a}{p^n}} f(z)\cdot (p^nz+a)^j dz\right.\\&\qquad\qquad\qquad\qquad\qquad\qquad\qquad\qquad\qquad +\chi(-1)\int_{i \infty}^{+\frac{a}{p^n}} f(z)\cdot (p^nz-a)^j dz\bigg)\\
    =& \frac{(-2\pi i )^j}{j!}\cdot p^{n(-j-1)}\cdot\tau(\overline{\chi})\cdot \frac{1}{2}\sum_{a\text{ mod }p^n}\chi(a)\cdot 2\pi i \cdot  2\Omega_f^{\pm}\cdot \left\langle \left\{-\frac{a}{p^n}-\infty\right\}, \left(\sum_{t=0}^jC_j^tp^{nt}a^{j-t}v_t\right)\otimes \per^{\pm}_F(f) \right\rangle \\
    &\text{\quad (Here the sign $\pm$ is chosen such that $\pm 1=\chi(-1)\cdot(-1)^j$.)}\\
    =&(-1)^{j}\cdot\frac{(2\pi i)^{j+1}}{\Gamma(j+1)}\cdot \frac{\Omega_f^{\pm}}{p^{n(j+1)}}\cdot \tau(\overline{\chi})\cdot\sum_{a\text{ mod }p^n}\chi(a)\cdot \left\langle \begin{pmatrix}p^n & a\\ 0 & 1
    \end{pmatrix}^{-1}\cdot\left\{0-\infty\right\}, \left(\begin{pmatrix}p^n & a\\ 0 & 1
    \end{pmatrix}^{-1}\cdot v_j\right)\otimes \per^{\pm}_F(f) \right\rangle \\
    =&(-1)^{j}\cdot\frac{(2\pi i)^{j+1}}{\Gamma(j+1)}\cdot \frac{\Omega_f^{\pm}}{p^{n(j+1)}}\cdot \tau(\overline{\chi})\cdot\sum_{a\text{ mod }p^n}\chi(a)\cdot \left\langle \left\{0-\infty\right\}, v_j\otimes \begin{pmatrix}p^n & a\\ 0 & 1
    \end{pmatrix}\per^{\pm}_F(f) \right\rangle \\
    &\text{\quad (Here we are using Remark \ref{2.2} about the action of $\begin{pmatrix}p^n & a\\ 0 & 1
    \end{pmatrix}$ on image(per).)}\\
    =&(-1)^{j}\cdot\frac{(2\pi i)^{j+1}}{\Gamma(j+1)}\cdot \frac{\Omega_f^{\pm}}{p^{n(j+1)}}\cdot \tau(\overline{\chi})\cdot\sum_{a\text{ mod }p^n}\chi(a)\cdot\mathcal{M}_f^{\pm}\left( v_{j}\otimes \adjustbox{scale=0.7}{$\begin{pmatrix}p^n & a\\ 0 & 1
    \end{pmatrix}$}v_{\new}\right).
\end{align*}

Here in the last equality, we are using Lemma \ref{lem-key} to identify the (global) action of the matrix $\begin{pmatrix}p^n & a\\ 0 & 1
    \end{pmatrix}$ on the modular symbols and the (local) action of it on $\Pi(V_f)$.

\vskip 1em

We write for all $0\leq j\leq k-2$, 
\begin{equation}\label{lambda}
    \widetilde{\Lambda}(f,\chi,j+1)=\frac{\Gamma(j+1)}{(2\pi i)^{j+1}}\cdot\frac{L(f,\chi,j+1)}{\Omega_f^{\pm}}
\end{equation}where the sign $\pm$ is chosen so that $ \chi(-1)\cdot (-1)^j=\pm 1$. Thus we have obtained the following lemma:

\begin{lemma}\label{lem-2.4}
For any finite order character $\chi$ of conductor $p^n$ and any integer $0\leq j\leq k-2$, we have \begin{equation}
    \widetilde{\Lambda}(f,\overline{\chi},j+1)= \frac{(-1)^{j}\tau(\overline{\chi})}{p^{n(j+1)}} \cdot\sum_{a\text{ mod }p^n}\chi(a)\cdot\mathcal{M}_f^{\pm}\left( v_{j}\otimes  \begin{pmatrix}p^n & a\\ 0 & 1
    \end{pmatrix}v_{\new}\right),
\end{equation}
where the sign $\pm$ is chosen to satisfy $ \chi(-1)\cdot (-1)^j=\pm 1$.
\end{lemma}
\begin{flushright}
$\Box$
\end{flushright}

\vskip 1em

We define $$F_{\chi}:=\sum_{a\text{ mod }p^n}\chi(a)\cdot\begin{pmatrix}p^n & a\\ 0 & 1\end{pmatrix}v_{\new}$$ when the conductor of $\chi$ is $p^n$. Then Lemma \ref{lem-2.4} can also be stated as:

\begin{lemma}\label{lem-2.5}
For any $0\leq j\leq k-2$, any finite order character $\chi$ of conductor $p^n$ with $n\geq 0$, let 
$$F_{\chi}:=\sum_{a\text{ mod }p^n}\chi(a)\cdot\begin{pmatrix}p^n & a\\ 0 & 1\end{pmatrix}v_{\new},$$
then we have for all $0\leq j\leq k-2$,

\begin{equation}
    \mathcal{M}_f^{\pm} \left( v_{j}\otimes F_{\chi}\right)=(-1)^{j}\cdot \frac{p^{n(j+1)}}{\tau(\overline{\chi})}\cdot \widetilde{\Lambda}(f,\overline{\chi},j+1).
\end{equation}

Here, the sign $\pm$ is chosen such that $$\chi(-1)=\pm (-1)^j.$$ 
\end{lemma}
\begin{flushright}
$\Box$
\end{flushright}

\begin{remark}\label{rem-2.6}
If the sign $\pm$ is such that $\chi(-1)\neq\pm (-1)^j$, then $\mathcal{M}_f^{\pm} \left( v_{j}\otimes F_{\chi}\right)=0$.
\end{remark}

\vskip 3em

\section{The \texorpdfstring{$p$}{p}-adic Kirillov model, supercuspidal case}

\noindent Throughout this section, we assume $\pi_p^{\sm}(f)$, the smooth $p$-adic representation of $\GL_2\left(\mQ_p\right)$ attached to the modular form $f$, is supercuspidal. We recall the theory of $p$-adic Kirillov models developed in \cite{C2}, chapter VI.

\vskip 1em

If $\Pi$ is a $p$-adic locally algebraic representation of $\GL_2(\mathbb{Q}_p)$ with coefficient field $L$ that can be written as $\Pi=\det^a\otimes\Sym^b\otimes\pi$ where $a$, $b\in\mathbb{Z}$, $b\geq 0$ and $\pi$ is a smooth representation of $\GL_2(\mathbb{Q}_p)$, then the $p$-adic Kirillov model of $\Pi$ is the unique $P(\mathbb{Q}_p)$-equivariant embedding $$\sK: \Pi\rightarrow \prod_{\mZ}t^aL_{\infty}[t]/t^{a+b+1},$$where the action of $P(\mathbb{Q}_p)$ on the RHS is given by the formula
\begin{equation}
    \left(\begin{pmatrix} a & b \\
    0 & 1
    \end{pmatrix}S\right)_n=\left(1\otimes\varepsilon\left(bp^n\right)\right)\cdot \exp\left(bp^nt\right)\cdot \left(1\otimes\sigma_{a^*}\right)\left(S_{v(a)+n}\right)
\end{equation}
for any $S\in \prod_{\mZ}t^aL_{\infty}[t]/t^{a+b+1}$ and any $n\in\mZ$. Here, $a^*:=|a|_p\cdot a$, and $\varepsilon$ is defined in the following way: we let $\varepsilon(x)=\zeta_{p^n}^r$ if $x\equiv \frac{r}{p^n}$ mod $\mZ_p$, where $\left\{\zeta_{p^n}\right\}_{n\geq 0}$ is the compatible system of $p$-power roots of unity chosen as in Section \ref{notation}.

\begin{remark}
The usual definition of the Kirillov model of a $p$-adic locally algebraic representation (as in \cite{C2}) is a $P(\mathbb{Q}_p)$ equivariant map $\mathscr{K}$ from $\Pi=\det^a\otimes\Sym^b\otimes\pi$ to $\LP^{[a,a+b]}\left(\mQ_p^*, t^aL_{\infty}[t]/t^{a+b+1}\right)$, where $L_{\infty}=L\otimes_{\mQ_p}\mQ_p(\zeta_{p^{\infty}})$. There is an action of $\Gamma$ on $\LP^{[a,a+b]}\left(\mQ_p^*, t^aL_{\infty}[t]/t^{a+b+1}\right)$ given by formula 
\begin{equation}
\sigma_u(S)(x)=(1\otimes\sigma_u)^{-1}\left(S(ux)\right)
\end{equation}
for any $u\in\mZ_p^*$ and $x\in\mQ_p^*$, which commutes with the action of $P(\mathbb{Q}_p)$ on $\LP^{[a,a+b]}\left(\mQ_p^*, t^aL_{\infty}[t]/t^{a+b+1}\right)$. Therefore the uniqueness of the $p$-adic Kirillov model forces the image of $\mathscr{K}$ to be contained in $\LP^{[a,a+b]}\left(\mQ_p^*, t^aL_{\infty}[t]/t^{a+b+1}\right)^{\Gamma}$, which can then be identified as $\prod_{\mZ}t^aL_{\infty}[t]/t^{a+b+1}$ by sending every $S\in\LP^{[a,a+b]}\left(\mQ_p^*, t^aL_{\infty}[t]/t^{a+b+1}\right)^{\Gamma}$ to $\{S(p^n)\}_{n\in\mZ}$. In the rest of this paper, we will always regard a $p$-adic Kirillov function as an element in  $\prod_{\mZ}t^aL_{\infty}[t]/t^{a+b+1}$.
\end{remark}

\vskip 1em

In the case when $\Pi=\Pi(V_f)^{\lalg}$ is the collection of locally algebraic vectors of the $p$-adic Banach space representation $\Pi(V_f)$ of $\GL_2(\mQ_p)$, we have $\Pi(V_f)^{\lalg}=\left(\Sym^{k-2}L^2\right)^{\vee}\otimes_L \pi_p^{\sm}(V_f)$, thus its $p$-adic Kirillov model is a $P(\mathbb{Q}_p)$-equivariant embedding $$\mathscr{K}: \Pi(V_f)^{\lalg}\rightarrow \prod_{\mathbb{Z}}\frac{1}{t^{k-2}}L_{\infty}[[t]]/tL_{\infty}[[t]],$$ where we continue using the notations $F=\mathbb{Q}(f)$,  $L=F_{\lambda}$ where $\lambda$ is a place lying over $p$.

\begin{remark}
Since in our case $\pi_p^{\sm}(f)$ is supercuspidal, the map $\mathscr{K}$ is in fact a $P(\mathbb{Q}_p)$-equivariant isomorphism onto $\bigoplus_{\mathbb{Z}}\frac{1}{t^{k-2}}L_{\infty}[[t]]/tL_{\infty}[[t]]$.
\end{remark}

We assume $\pi_p^{\sm}(f)$ has  central character $\delta$. Then the explicit formula of the Borel action on $\prod_{\mathbb{Z}}\frac{1}{t^{k-2}}L_{\infty}[[t]]/tL_{\infty}[[t]]$ is: 
\begin{align*}
    \left(\begin{pmatrix} a & b \\
    0 & d
    \end{pmatrix}S\right)_n&=\left(d^{2-k}\delta(d)\otimes\varepsilon\left(\frac{bp^n}{d}\right)\right)\cdot \exp\left(\frac{bp^nt}{d}\right)\cdot \left(1\otimes\sigma_{\left(\frac{a}{d}\right)^*}\right)\left(S_{v(a)-v(d)+n}\right)\\
    &=d^{2-k}\delta(d)\cdot \left[(1+T)^{\frac{bp^n}{d}}\right]\cdot \left(1\otimes\sigma_{\left(\frac{a}{d}\right)^*}\right)\left(S_{v(a)-v(d)+n}\right).
\end{align*}

If we write the weight vectors of $\left(\Sym^{k-2}L^2\right)^{\vee}=\Sym^{k-2}L^2\otimes_L\det^{2-k}$ to be $$v_0=v_{\hw},\text{ } v_1,\text{ } \cdots,\text{ } v_{k-3},\text{ } v_{k-2}=v_{\text{lw}}$$ as in the previous sections, then by \cite{C2} Page 146, we have \begin{equation}
    \left(\mathscr{K}\left( v_j\otimes v\right)\right)_n=j!p^{-nj}\cdot \left(\mathscr{K}^{\sm}(v)\right)_n\cdot t^{-j},
\end{equation}
where $\mathscr{K}^{\sm}$ is the $p$-adic Kirillov model for smooth representations.

\vskip 1em

For every integer $n$, we let $e_n\in\prod_{\mathbb{Z}}L_{\infty}$ be the vector whose coordinate at the $n$-th place is $1$, and $0$ everywhere else.

Recall that we are in the case when $\pi_p^{\sm}(f)$ is supercuspidal, hence $\mathscr{K}^{\sm}(v_{\new})=\mathbbm{1}_{\mathbb{Z}_p^*}=e_0$. Thus 
\begin{equation}
    \mathscr{K}^{\sm}\left(F_{\chi}\right)=\left(\sum_{a\text{ mod }p^n}\chi(a)\otimes\zeta_{p^n}^a\right)\cdot e_{-n},
\end{equation}
where $\chi$ is any finite order character of conductor $p^n$ with $n\geq 0$, and $F_{\chi}$ is defined as in section \ref{sec-1}.

Therefore, we have for all $0\leq j\leq k-2$, \begin{equation}\label{4.7}
    \mathscr{K}(v_j\otimes F_{\chi})=j!\cdot p^{nj}\cdot \left(\sum_{a\text{ mod }p^n}\chi(a)\otimes\zeta_{p^n}^a\right)\cdot t^{-j}\cdot e_{-n}.
\end{equation}

\vskip 1em

In the rest of this section, we will describe the $p$-adic Kirillov model of $\Pi(V_f)$ from another point of view.

\vskip 1em

Recall the following theorem from \cite{C2}:

\begin{theorem}[Colmez, \cite{C2}, Corollaire II.2.9]

If $\pi_p^{\sm}(f)$ is supercuspidal, then there is an isomorphism as $P\left(\mQ_p\right)$-modules:
\begin{equation}
\DD\Big/\dd\xrightarrow{\cong}\Pi(V_f).
\end{equation}
\end{theorem}

\begin{remark} The action of $P\left(\mQ_p\right)$ on $\DD\Big/\dd$ is defined as follows:
\begin{itemize}
 \item The matrix $\begin{pmatrix} p & 0 \\ 0 & 1
\end{pmatrix}$ acts on $\DD\Big/\dd$ via the operator $\varphi$;
 \item The matrix $\begin{pmatrix} \mZ_p^* & 0 \\ 0 & 1
\end{pmatrix}$ acts on $\DD\Big/\dd$ via the operator $\Gamma$;
 \item The matrix $\begin{pmatrix} 1 & a \\ 0 & 1
\end{pmatrix}\in \begin{pmatrix} 1 & \mQ_p \\ 0 & 1
\end{pmatrix}$ acts on $\DD\Big/\dd$ via multiplying $\left(1+T\right)^a$.
\end{itemize}
\end{remark}

\vskip 1em

We denote  the inverse of the above isomorphism by $\eta$. We then have the following lemma:

\begin{lemma}\quad

The image of $\Pi(V_f)^{\lalg}$ under the map $\eta$ is contained in $\DDd$.
\end{lemma}

\begin{proof}
For any $v_{\sm}\in\pi_p^{\sm}(f)$, $ v_0\otimes v_{\sm}$ is fixed by $\begin{pmatrix} 1 & p^r \\
    0 & 1
    \end{pmatrix}$ if $r$ is sufficiently large. This means $$(1+T)^{p^r}\cdot\eta( v_0\otimes v_{\sm})=\eta(v_0\otimes v_{\sm}),$$that is, $\eta( v_0\otimes v_{\sm})\in \frac{1}{\varphi^r(T)}\dd\Big/\dd$. 
    
We then proceed by induction: if $0\leq j<k-2$ and $\eta(v_i\otimes v_{\sm})\in \frac{1}{\left(\varphi^r(T)\right)^{j}}\dd\Big/\dd$ for all $0\leq i<j$ and some $r$, then we will show that $\eta(v_j\otimes v_{\sm})\in \frac{1}{\left(\varphi^{r}(T)\right)^{j+1}}\dd\Big/\dd$. 

To see this, notice that for sufficiently large $r$, $$\begin{pmatrix} 1 & p^r \\
    0 & 1
    \end{pmatrix}(v_j\otimes v_{\sm})=\begin{pmatrix} 1 & p^r \\
    0 & 1
    \end{pmatrix}v_j\otimes v_{\sm}= v_j\otimes v_{\sm}+\sum_{0\leq i<j}\begin{pmatrix} j \\
    i
    \end{pmatrix}p^{r(j-i)} v_i\otimes v_{\sm},$$
so $$ \varphi^r(T)\cdot \eta(v_j\otimes v_{\sm})\in \frac{1}{\left(\varphi^r(T)\right)^{j}}\dd\Big/\dd$$and hence $\eta(v_j\otimes v_{\sm})\in \frac{1}{\left(\varphi^{r}(T)\right)^{j+1}}\dd\Big/\dd$.
\end{proof}

\vskip 1em

\begin{definition}
We define a map $$\sI:\text{ }\DDd\longrightarrow\prod_{\mZ}\widetilde{\D}_{\dif}(V_f(1))\big/\widetilde{\D}^+_{\dif}(V_f(1))$$ as follows:

\begin{itemize}
 \item The $0$-th coordinate $\sI_0:=\iota_0^-$, which is the natural map  $$\iota_0^-:\text{ }\DDd\longrightarrow \widetilde{\D}_{\dif}(V_f(1))
\big/\widetilde{\D}^+_{\dif}(V_f(1))$$ induced by the inclution $\Bb\subset\B_{\dR}$.
 \item The $n$-th coordinate $\sI_n$ is defined as $\iota_0^-\circ\varphi^{n}$ if $n\geq 0$.
 \item The $(-n)$-th coordinate $\sI_n$ is defined as $\iota_{n}^-$ if $n\geq 0$: we choose a compatible system of $p$-power roots of unity $\left\{\zeta_{p^i}\right\}_{i\geq 0}$ as before, and for any positive integer $n$ and any function $f(T)\in \widetilde{\B}^+$, we define $$\iota_n\left(f(T)\right):=f\left(\zeta_{p^n}\exp\left(\frac{t}{p^n}\right)-1\right),$$ The map $\iota_n$ extends naturally to a map from $\Bb$ to $\B_{\dR}$, which induces a map $\iota_n$ from $\Dd$ to ${\widetilde{\D}_{\dif}(V_f(1))}$. The map $\iota_n^-$ is the composite of the natural projection from ${\widetilde{\D}_{\dif}(V_f(1))}$ to $\widetilde{\D}_{\dif}(V_f(1))\big/\widetilde{\D}^+_{\dif}(V_f(1))$ with $\iota_n$.
\end{itemize}
\end{definition}

\vskip 1em

We equip $\prod_{\mZ}{\widetilde{\D}_{\dif}(V_f(1))}\big/{\widetilde{\D}^+_{\dif}(V_f(1))}$ with an action of $P(\mQ_p)$ in the following way: 
\begin{itemize}
    \item The matrix $\begin{pmatrix} \mZ_p^* & 0 \\
    0 & 1
    \end{pmatrix}$ acts coordinate wise through the action of $\Gamma$;
    \item A matrix $\begin{pmatrix} 1 & b \\
    0 & 1
    \end{pmatrix}\in\begin{pmatrix} 1 & \mQ_p \\
    0 & 1
    \end{pmatrix}$ acts on the $n$-th coordinate by multiplying $\varepsilon\left(bp^n\right)\cdot \exp\left(bp^nt\right)$;
    \item $\begin{pmatrix} p & 0 \\
    0 & 1
    \end{pmatrix}$ acts by shifting to the left.
\end{itemize}

We then have the following observation: 

\begin{prop}\label{prop-2.5}
The map $$\sI:\text{ }\DDd\longrightarrow\prod_{\mZ}{\widetilde{\D}_{\dif}(V_f(1))}
\big/{\widetilde{\D}^+_{\dif}(V_f(1))}$$ is $P(\mQ_p)$-equivariant. Moreover, the following diagram is commutative:
\begin{center}
\begin{tikzcd}
\DDd  \ar[r, "\sI"] &     \prod_{\mZ}{\widetilde{\D}_{\dif}(V_f(1))}
\big/{\widetilde{\D}^+_{\dif}(V_f(1))}   \\
\Pi(V_f)^{\lalg} \ar[r, "\sK"]\ar[u,hook,"\eta"] &   \prod_{\mathbb{Z}}\frac{1}{t^{k-2}}L_{\infty}[[t]]/tL_{\infty}[[t]] \ar[u,hook]
\end{tikzcd}
\end{center}
where the vertical map on the right on each coordinate is given by the composite of the following maps: 
\begin{align*}
    \frac{1}{t^{k-2}}L_{\infty}[[t]]/tL_{\infty}[[t]] \xrightarrow{\text{ }\cong\text{ }}\D_{\dR}(V_f(1))\otimes_L \frac{1}{t^{k-2}}L_{\infty}[[t]]\Big/{\D^+_{\dif}(V_f(1))}& \hookrightarrow {\D_{\dif}(V_f(1))}
/{\D^+_{\dif}(V_f(1))}\\
&\subset {\widetilde{\D}_{\dif}(V_f(1))}
\big/{\widetilde{\D}^+_{\dif}(V_f(1))}.
\end{align*}
The first isomorphism above is defined by sending every $g(t)$ to $\overline{f}^*\otimes g(t)$. See Section \ref{notation} for the precise definition of $\overline{f}^*\in \D_{\dR}(V_f(1))/\Fil^0\D_{\dR}(V_f(1))$.
\end{prop}

\begin{proof}
The $P(\mQ_p)$-equivariance of the map $\sI$ can be checked by direct computation. 

The commutativity of the diagram follows from the uniqueness of the $p$-adic Kirillov model.
\end{proof}

\vskip 1em

\begin{cor}\label{cor-3.6} Let $\chi$ be any finite order character of conductor $p^n$, where $n\geq 0$. Let $F_{\chi}$ be as defined in section \ref{sec-1}. Then $\sI_m\circ\eta\left(F_{\chi}\otimes v_j\right)=0$ for any $m\neq -n$, and $\sI_{-n}\circ\eta\left(F_{\chi}\otimes v_j\right)$ is contained in $\D_{\dif,n}\left(V_f(1)\right)\big/\D_{\dif,n}^+\left(V_f(1)\right)$, for any $0\leq j\leq k-2$.

In other words, we have:
\begin{itemize}
	\item[--] $\iota^-_m\circ\eta\left(v_j\otimes  F_{\chi}\right)=0$ for all integer $m\neq n$, 
    \item[--] $\iota^-_n\circ\eta\left(v_j\otimes  F_{\chi}\right)\in \D_{\dif,n}\left(V_f(1)\right)\big/\D_{\dif,n}^+\left(V_f(1)\right)$.
\end{itemize}
\end{cor}

\begin{proof}
This follows immediately from Proposition \ref{prop-2.5}, together with the explicit formula of the map $\sK$ in equation \eqref{4.7}.
\end{proof}

\vskip 3em

\section{An explicit reciprocity law, supercuspidal case}

\noindent In this section, we continue assuming that the ($p$-adic) smooth representation of $\GL_2\left(\mQ_p\right)$ attached to the modular form $f$ is supercuspidal. We review an ``explicit reciprocity law" proved in \cite{D} and \cite{EPW} as a generalization of Proposition VI.3.4 in \cite{C2}.

\vskip 1em

We first introduce two pairings that will be used later. The reference for the following definitions is Page 151 of \cite{C2}.

\subsubsection*{The pairing ``$\{\text{ },\text{ }\}$"}\quad

\vskip 1em

The pairing $$V_f(1)\times V_f^*\rightarrow L(1)$$ gives rise to the following pairing: $$[\text{ },\text{ }]:\text{\quad}\widetilde{\D}\left(V_f(1)\right)\times\widetilde{\D}( V_f^*)\rightarrow \widetilde{\D}(L(1))\cong \widetilde{B}_L.$$

We define $\{\text{ },\text{ }\}$ to be $\Res_{T=0}\left([\text{ },\text{ }]\cdot\frac{dT}{1+T}\right)$. 

The key properties of the pairing  $\{\text{ },\text{ }\}$ are summarized in the following proposition, whose proof can be found in \cite{C2}. 

\begin{prop}\quad

\begin{enumerate}
    \item $\widetilde{\D}\left(V_f(1)\right)/\widetilde{\D}^+\left(V_f(1)\right)$ and\quad$\Big(\D^{\natural}( V_f^*)\boxtimes\mQ_p\Big)_{\bb}:=\Big(\lim\limits_{\substack{\longleftarrow \\ \psi}}\D^{\natural}( V_f^*)\Big)_{\bb}$ are topological dual to each other. We have $\widetilde{\D}^+\left( V_f^*\right)=\left(\D^{\natural}( V_f^*)\boxtimes\mQ_p\right)_{\text{\normalfont{pc}}}\subset \left(\D^{\natural}( V_f^*)\boxtimes\mQ_p\right)_{\bb}$, and the pairing between $\widetilde{\D}\left(V_f(1)\right)/\widetilde{\D}^+\left(V_f(1)\right)$ and\quad$\left(\D^{\natural}( V_f^*)\boxtimes\mQ_p\right)_{\bb}$ restriced to $\widetilde{\D}^+\left( V_f^*\right)=\left(\D^{\natural}( V_f^*)\boxtimes\mQ_p\right)_{\text{\normalfont{pc}}}$ is the pairing $\{\text{ },\text{ }\}$ defined above.
    \item We have the following compatibility:
    \begin{center}
    \begin{tikzcd}
      \widetilde{\D}^+\left(V_f(1)\right)\left[\frac{1}{\varphi^i(T)}\text{, }_{i\geq 0}\right]/\widetilde{\D}^+(V_f(1)) \ar[r, phantom, "\times"] &  \D(V_f^*)^{\psi=1}\ar[r,"\{\text{ , }\}"] &L\\
  \Pi(V_f)^{\lalg}\ar[u, hook, "\eta"]\ar[r, phantom, "\times"] & \Pi(V_f)^{*,\pppp=1}\ar[u, "\fC", "\cong"']\ar[r] &L\ar[u,equal]
    \end{tikzcd}
    \end{center}
    where the pairing on the bottom row is induced by the usual pairing between $\Pi(V_f)$ and $\Pi(V_f)^{*}$.
\end{enumerate}
\end{prop}

\begin{proof}
(1) is \cite{C2} Proposition I.3.20. and Proposition I.3.21.
\end{proof}

\vskip 1em

\subsubsection*{The pairings ``$\langle \text{ },\text{ }\rangle_{\dif}$" and ``$\langle \text{ },\text{ }\rangle_{\dif,m}$"}\quad

\vskip 1em

The pairing $V_f(1)\times V_f^*\rightarrow L(1)$ again gives rise to the following pairings (we still use the notation ``$[\text{ },\text{ }]$" here): $$[\text{ },\text{ }]:\text{\quad}\widetilde{\D}_{\dif}\left(V_f(1)\right)\times\widetilde{\D}_{\dif}\left( V_f^*\right)\rightarrow \widetilde{\D}_{\dif}(L(1))\cong B_{\dR}^H\otimes_{\mQ_p}L;$$and for every integer $m\geq 0$,
$$[\text{ },\text{ }]_m:\text{\quad}\D_{\dif,m}\left(V_f(1)\right)\times\D_{\dif,m}\left( V_f^*\right)\rightarrow \D_{\dif,m}(L(1))\cong L_m((t)).$$

We define $\langle \text{ },\text{ }\rangle_{\dif}$ (resp. $\langle \text{ },\text{ }\rangle_{\dif,m}$) as $\Tr\circ\Res_{t=0}\left([\text{ },\text{ }]\cdot\frac{dt}{t}\right)$ (resp. $\Tr\circ\Res_{t=0}\left([\text{ },\text{ }]_m\cdot\frac{dt}{t}\right)$), where $\Tr$ is the normalized trace map. 

The key properties of the pairings $\langle \text{ },\text{ }\rangle_{\dif}$ and $\langle \text{ },\text{ }\rangle_{\dif,m}$ are summarized in the following proposition, whose proof can again be found in \cite{C2}, Chapter VI, Section 3.4. 

\begin{prop}\label{4.2}\quad

\begin{enumerate}
    \item The pairings $\langle \text{ },\text{ }\rangle_{\dif}$ and $\langle \text{ },\text{ }\rangle_{\dif,m}$ are compatible for every integer $m\geq 0$. In other words, we have the following commutative diagram:
         \begin{center}
    \begin{tikzcd}
     \D_{\dif,m}\left(V_f(1)\right)\ar[d,hook] \ar[r, phantom, "\times"] &  \D_{\dif,m} ( V_f^* )\ar[d,hook] \ar[rr, "\langle\text{ , }\rangle_{\dif,m}"] &&L\\
  \widetilde{\D}_{\dif}\left(V_f(1)\right) \ar[r, phantom, "\times"]  & \widetilde{\D}^+_{\dif}( V_f^*)\ar[rr, "\langle\text{ , }\rangle_{\dif}"] &&L\ar[u,equal]
    \end{tikzcd}
    \end{center}
    \item $\widetilde{\D}_{\dif}^+\left(V_f(1)\right)$ and $\widetilde{\D}^+_{\dif}( V_f^*)$ are orthogonal complement of each other under the pairing $\langle \text{ },\text{ }\rangle_{\dif}$.
    \item For every integer $m\geq 0$, $\D_{\dif,m}^+\left(V_f(1)\right)$ and $\D_{\dif,m}^+( V_f^*)$ are orthogonal compliment of each other under the pairing $\langle \text{ },\text{ }\rangle_{\dif,m}$.
    \item We have the following compatibility:
    \begin{center}
    \begin{tikzcd}
      \D_{\dR}(V_f(1))\otimes_{\mQ_p}\B_{\dR}^H \ar[r, phantom, "\times"] &  \D_{\dR}(V_f^*)\otimes_{\mQ_p}\B_{\dR}^H \ar[rr, "\widetilde{\sB}"]  &&L\\
  \widetilde{\D}_{\dif}\left(V_f(1)\right)\ar[u, "\cong"] \ar[r, phantom, "\times"]  & \widetilde{\D}_{\dif}( V_f^*)\ar[u, "\cong"]\ar[rr, "\langle\text{ , }\rangle_{\dif}"] &&L\ar[u,equal]
    \end{tikzcd}
    \end{center}
    where the pairing $\widetilde{\sB}$ on the top row is defined by the formula 
    $$\widetilde{\sB}(x\otimes f(t), y\otimes g(t)):=\Tr\circ\Res_{t=0}\left([x,y]_{\dR}\cdot f(t)g(t) dt\right),$$and $[\text{ , }]_{\dR}$ is the pairing $$\D_{\dR}(V_f(1))\times  \D_{\dR}(V_f^*)\xrightarrow{[\text{ , }]_{\dR}}\D_{\dR}(L(1))=\frac{1}{t}L.$$
    \item Similarly, we have the following compatibility for every integer $m\geq0$:
    \begin{center}
    \begin{tikzcd}
      \D_{\dR}(V_f(1))\otimes_{L}L_m((t)) \ar[r, phantom, "\times"] &  \D_{\dR}(V_f^*)\otimes_{L}L_m((t)) \ar[rr, "\sB_m"]  &&L\\
  \D_{\dif,m}\left(V_f(1)\right)\ar[u, "\cong"] \ar[r, phantom, "\times"]  & \D_{\dif,m}( V_f^*)\ar[u, "\cong"]\ar[rr, "\langle\text{ , }\rangle_{\dif,m}"] &&L\ar[u,equal]
    \end{tikzcd}
    \end{center}
    where the pairing $\sB_m$ on the top row is defined by the formula 
    $$\sB_m(x\otimes f(t), y\otimes g(t)):=\Tr\circ\Res_{t=0}\left([x,y]_{\dR}\cdot f(t)g(t) dt\right).$$
\end{enumerate}
\end{prop}

\vskip 1em

The following theorem, which we shall call the ``explicit reciprocity law" relating the pairing $\langle\text{ , }\rangle_{\dif}$ and the pairing $\{\text{ , }\}$, is proved by Dospinescu:

\begin{theorem}[Dospinescu \cite{D}, Th\'{e}or\`{e}me 8.3.1]\label{dos} Let $\Pi$ be a continuous Banach space representation of $\GL_2(\mQ_p)$, $\Pi^{\lalg}$ be the collection of locally algebraic vectors of $\Pi$. Assume $\Pi^{\lalg}=W\otimes\pi$ with $W$ algebraic and $\pi$ a supercuspidal smooth representation of $\GL_2(\mQ_p)$. Let $v\in\Pi^{\lalg}$ with $\sK_j(v)=0$ for all but finitely many $j$. Assume $m$ is sufficiently large and $\sI_j\circ\eta(v)\in\D_{\dif,m}/\D_{\dif,m}^+$ for all $j$ (using the notations as in Proposition \ref{prop-2.5}). Then for any $z\in \Pi^{*,\pool}$, we have $$\{v,z\}=\sum_{j\in\mZ}\left\langle\sI_j\circ\eta(v)\text{, }\iota_m(\mathfrak{C}(z))\right\rangle_{\dif,m}.$$ In particular, the RHS of the above equation is independent of $m$ (as long as $m$ is sufficiently large).
\end{theorem}\begin{flushright}
$\Box$
\end{flushright}

Recall that we have defined the element $\cM^{\pm}_f\in \Pi(V_f)^{*,\pppp=1}$ and $\z_M^{\pm}(f)=\left(\Exp^*\right)^{-1}\fC(\cM_f^{\pm})\in\HH^1_{\Iw}(\mQ_p,V_f^*)$. Since $\Exp^*(\z^{\pm}_M)\in \D(V_f^*)^{\psi=1}$, $\iota_n$ is defined on $\Exp^*(\z^{\pm}_M)$ for all $n\geq 1$ and $$\iota_n\circ\Exp^*(\z^{\pm}_M)\in\D_{\dif,n}^+(V_f^*).$$ Moreover, one has \begin{equation}
   \frac{1}{p}\Tr_{L_n}^{L_{n+1}}\iota_{n+1}\circ\Exp^*(\z^{\pm}_M)=\iota_n\circ\Exp^*(\z^{\pm}_M)
 \end{equation}
for all $n\geq 1$.

\vskip 1em

\begin{cor}\label{CEPW}
Let $\chi$ be a finite order character of $\mZ_p^*$ with conductor $p^n$, $n\geq0$. $F_{\chi}\in\pi_p^{\sm}(f)$ be as in section \ref{sec-1}.  For every $0\leq j\leq k-2$, \begin{equation}
 \mathcal{M}_f^{\pm}\left(v_j\otimes F_{\chi}\right)=\begin{cases}
\left\langle\iota_n^-\circ\eta(v_j\otimes F_{\chi})\text{ , }\iota_n\circ \Exp^*(\z_M^{\pm})\right\rangle_{\dif,n}, & \text{if }n\geq 1;\\
\left\langle\iota_0^-\circ\eta(v_j\otimes F_{\chi})\text{ , }\frac{1}{p-1}\Tr^{L_1}_{L_0}\iota_1 \circ\Exp^*(\z_M^{\pm})\right\rangle_{\dif,0}, & \text{if }n=0.
\end{cases}
 \end{equation}
\end{cor}

\begin{proof}
Using Theorem \ref{CEPW} and Corollary \ref{cor-3.6}, we conclude that $$\mathcal{M}_f^{\pm}\left(v_j\otimes F_{\chi}\right)=\left\langle \iota_n^-\circ\eta(v_j\otimes F_{\chi})\text{ , }\iota_m\circ\Exp^*(\z_M^{\pm})\right\rangle_{\dif,m}$$ for all $m$ sufficiently large, where we view  $\iota_n^-\circ\eta(v_j\otimes F_{\chi})$ as an element in $\D_{\dif,m}(V_f(1))/\D_{\dif,m}^+(V_f(1))$ via the natural inclusion $$\D_{\dif,n}(V_f(1))/\D_{\dif,n}^+(V_f(1))\hookrightarrow \D_{\dif,m}(V_f(1))/\D_{\dif,m}^+(V_f(1)).$$ 
Since $$\left\langle \iota_n^-\circ\eta(v_j\otimes F_{\chi})\text{ , }\iota_m\circ\Exp^*(\z_M^{\pm})\right\rangle_{\dif,m}=\left\langle \iota_n^-\circ\eta(v_j\otimes F_{\chi})\text{ , }\Tr_{L_n}\iota_m\circ\Exp^*(\z_M^{\pm})\right\rangle_{\dif,n}$$ where $\Tr_{L_n}$ denotes the normalized trace map to $L_n$, and $$\Tr_{L_n}\iota_m\circ\Exp^*(\z_M^{\pm})=\iota_n\circ\Exp^*(\z_M^{\pm})$$ for all $n\geq 1$, the conclusion follows.
\end{proof}

\vskip 1em

Notice that $V_f^*$ has Hodge-Tate weights $0$ and $k-1$, we have $$\D_{\dif,n}^+\left(V_f^*\right)=\Fil^0\left(\D_{\dR}(V_f^*)\otimes_{L}L_n((t))\right)=\D_{\dR}(V_f^*)\otimes_{L}t^{k-1}L_n[[t]]+\Fil^0\D_{\dR}(V_f^*)\otimes_{L}L_n[[t]].$$

We make the following definition:

\begin{definition}\label{def-5.2} Let $\overline{f}\in\Fil^0\left(\D_{\dR}(V_f^*)\right)$ be defined as in Section \ref{notation}.

 For any $0\leq j\leq k-2$ and any integer $r\geq 1$, we write $C_{r,j}^{\pm}\in L_r= L\otimes_{\mathbb{Q}_p}\mathbb{Q}_p\left(\zeta_{p^r}\right)$ the coefficient of $\iota_r\circ\Exp^*(\z^{\pm}_M)$ in front of $t^j$: 
\begin{equation}
\iota_r\circ\Exp^*(\z^{\pm}_M)=\overline{f}\otimes\sum_{j=0}^{k-2}C_{r,j}^{\pm}t^j+\text{ ``Something in $\D_{\dR}(V_f^*)\otimes_{L}t^{k-1}L_r[[t]]$".} 
\end{equation} 

Further more, under the decomposition
$$
L\otimes_{\mathbb{Q}_p}\mathbb{Q}_p\left(\zeta_{p^r}\right) = \bigoplus_{\phi} L(\phi^{-1})$$
where the direct sum is taken among all characters $\phi$ of $\left(\mathbb{Z}/p^r\right)^{\times}$, we write \begin{equation}
C_{r,j}^{\pm} = \sum_{\phi} C_{r,j,\phi}^{\pm} \cdot e_{\phi}
 \end{equation} with $C_{r,j,\phi}^{\pm}\in L$.
 
 Here, we are viewing each $L(\phi^{-1})$ as a 1-dimensional $L$-subspace contained in $L\otimes_{\mathbb{Q}_p}\mathbb{Q}_p\left(\zeta_{p^r}\right)$, with basis $$e_{\phi}=\sum_{a\text{ mod }p^r}\phi(a)\otimes\zeta_{p^l}^a\in L\otimes_{\mathbb{Q}_p}\mathbb{Q}_p\left(\zeta_{p^r}\right),$$ where $p^l$ is the conductor of $\phi$.
\end{definition}

\vskip 1em

\begin{remark}
Since $\Exp^*(\z^{\pm}_M)$ is fixed by $\psi$, we have $$C_{r,j}^{\pm}=\frac{1}{p}\Tr_{\mQ_{p,r}}^{\mQ_{p,r+1}}C_{r+1,j}^{\pm}$$ for all integers $r\geq 1$ and $j\geq 0$. 
\end{remark}

\begin{remark}It can be checked easily that 
$$C_{r,j,\phi}^{\pm}=\frac{1}{p^{r-1}(p-1)}\sum_{a\text{ mod }p^r}\phi(a)\sigma_a(C_{r,j}^{\pm})$$for all integers $r\geq 1$, $j\geq 0$ and $\phi$ a character of $\left(\mZ/p^r\right)^{\times}$.
\end{remark}

\vskip 1em

We define \begin{equation}
 C_{0,j}^{\pm}=\frac{1}{p-1}\Tr_{\mQ_{p}}^{\mQ_{p}(\zeta_p)}C_{1,j}^{\pm}
 \end{equation}
for all $j\geq 0 $.

\begin{prop}\label{prop-4.9}
Let $\chi$ be any finite order character of conductor $p^n$, $n\geq 0$. 

We have for all $0\leq j\leq k-2$, 
\begin{equation}
    \Tr\circ\Res_{t=0}\left(\sK_{-n}\left(v_j\otimes F_{\chi}\right)\cdot \sum_{i=0}^{k-2}C_{n,j}^{\pm}t^i\cdot\frac{dt}{t}\right)=\cM_f^{\pm}\left(v_j\otimes F_{\chi}\right),
\end{equation}
where $\Tr$ is the normalized trace map to $L$.
\end{prop}

\begin{proof}
By Theorem \ref{dos}, for any integer $m$ sufficiently large, $$\langle \iota_n^-\circ\eta(v_j\otimes F_{\chi})\text{ , }\iota_m\circ\fC\left(\cM_f^{\pm}\right)\rangle_{\dif,m}=\cM_f^{\pm}\left(v_j\otimes F_{\chi}\right).$$

Now by Proposition \ref{prop-2.5}, $$\iota_n^-\circ\eta(v_j\otimes F_{\chi})=\overline{f}^*\otimes\sK_{-n}\left(v_j\otimes F_{\chi}\right)\in \frac{\D_{\dR}(V_f(1))}{\Fil^0\D_{\dR}(V_f(1))}\otimes_{L}\frac{t^{2-k}L_n[[t]]}{tL_n[[t]]};$$ and by Definition \ref{def-5.2}, \begin{align*}
    \iota_m\circ\Exp^*(\z^{\pm}_M)=\iota_m\circ\fC\left(\cM_f^{\pm}\right)&=\overline{f}\otimes\sum_{i=0}^{k-2}C_{m,i}^{\pm}t^j+\text{ ``Something in $\D_{\dR}(V_f^*)\otimes_{L}t^{k-1}L_m[[t]]$"} \\
    &\in \Fil^0\D_{\dR}(V_f^*)\otimes_{L}L_m[[t]]+\D_{\dR}(V_f^*)\otimes_{L}t^{k-1}L_m[[t]]
\end{align*}
whenever $m\geq 1$.

Notice that anything in $\D_{\dR}(V_f^*)\otimes_{L}t^{k-1}L_m[[t]]$ paired with $\overline{f}^*\otimes\sK_{-n}\left(v_j\otimes F_{\chi}\right)$ equals $0$ because the latter element has coefficient $0$ in front of $t^i$ for all $i\leq 1-k$. Thus
\begin{align*}
    \left\langle \iota_n^-\circ\eta(v_j\otimes F_{\chi})\text{ , }\iota_m\circ\fC\left(\cM_f^{\pm}\right)\right\rangle_{\dif,m}&=\left\langle \overline{f}^*\otimes\sK_{-n}\left(v_j\otimes F_{\chi}\right)\text{ , }\Tr_{L_{n}}\left(\iota_m\circ\fC\left(\cM_f^{\pm}\right)\right)\right\rangle_{\dif,n}\\
    &=  \left\langle \overline{f}^*\otimes\sK_{-n}\left(v_j\otimes F_{\chi}\right)\text{ , }\overline{f}\otimes\sum_{i=0}^{k-2}C_{n,i}^{\pm}t^j\right\rangle_{\dif,n}   \\
    &=\Tr\circ\Res_{t=0}\left([\overline{f}^*,\overline{f}]_{\dR}\cdot\sK_{-n}\left(v_j\otimes F_{\chi}\right)\cdot \sum_{i=0}^{k-2}C_{n,i}^{\pm}t^j \cdot dt\right)\\
    &=\Tr\circ\Res_{t=0}\left(\sK_{-n}\left(v_j\otimes F_{\chi}\right)\cdot \sum_{i=0}^{k-2}C_{n,j}^{\pm}t^i\cdot\frac{dt}{t}\right).
\end{align*}
Here, $\Tr_{L_n}$ denotes the normalized trace map to $L_n$, and $\Tr$ is the normalized trace map to $L$.
\end{proof}

\vskip 1em

\begin{cor}\label{5.7}
Let $\chi$, $j$ be as in Propostion \ref{prop-4.9}. We have: 
\begin{align*}
   j!\cdot p^{nj}\cdot \bigg(\Id\otimes\Tr\bigg)\left( C^{\pm}_{n,j}\cdot \sum_{a\text{ mod }p^n}\chi(a)\otimes\zeta_{p^n}^a\right)&=\mathcal{M}_f^{\pm}\left(v_{j}\otimes F_{\chi}\right)\\
   &=(-1)^{j}\cdot \frac{p^{n(j+1)}}{\tau(\overline{\chi})}\cdot \widetilde{\Lambda}(f,\overline{\chi},j+1)
\end{align*} 
when the sign $\pm$ on the LHS are chosen so that $\chi(-1)=\pm (-1)^j$, and  \begin{align*}
   j!\cdot p^{nj}\cdot  \bigg(\Id\otimes\Tr\bigg)\left( C^{\pm}_{n,j}\cdot \sum_{a\text{ mod }p^n}\chi(a)\otimes\zeta_{p^n}^a\right)&=0
\end{align*} 
if the sign $\pm$ doesn't satisfy $\chi(-1)=\pm (-1)^j$. Here, $\Tr$ denotes the normalized trace map to $\mQ_p$.
\end{cor}

\vskip 1em

\begin{proof}
This follows directly from Proposition \ref{prop-4.9}, equation \eqref{4.7}, Lemma \ref{lem-2.5} and Remark \ref{rem-2.6}.
\end{proof}

\vskip 1em

\begin{cor}\label{5.3}

Let $\chi$ be any finite order character of conductor $p^n$, $n\geq 1$ as before, and $C_{n,j,\chi}^{\pm}$ be as defined in Definition \ref{def-5.2}. Then for every $0\leq j\leq k-2$,
\begin{equation}\label{5.8}
   C_{n,j,\chi}^{\pm}=\pm \frac{p}{p-1}\cdot\frac{1}{\tau(\chi)}\cdot \frac{\widetilde{\Lambda}(f,\chi,j+1)}{j!}
 \end{equation}
 when the sign $\pm$ are chosen so that $\chi(-1)=\pm (-1)^j$. 
 
 \vskip 0.5em

 If the sign $\pm$ doesn't satisfy $\chi(-1)=\pm (-1)^j$, then $C_{n,j,\chi}^{\pm}=0$. 
\end{cor}

\vskip 1em

\begin{proof} For any character $\phi$ of $\left(\mathbb{Z}/p^n\right)^*$ with conductor $p^l$ ($l\leq n$), using the notation in Definition \ref{def-5.2}, we have a basis vector $e_{\phi}$ of the $1$-dimensional vector space $L(\phi^{-1})$: $$e_{\phi}:=\sum_{a\text{ mod }p^n}\phi(a)\otimes\zeta_{p^l}^a\in L(\phi^{-1})\subset L\otimes_{\mathbb{Q}_p}\mathbb{Q}_p\left(\zeta_{p^n}\right).$$

We can check easily that \begin{equation}
 \bigg(\Id\otimes \Tr \bigg)(e_{\phi})=\begin{cases}
p^{n-1}(p-1), & \text{If $\phi=1$ is the trivial character;}\\
\text{\quad} 0\text{\quad} , & \text{If $\phi$ has conductor $p^l$ with $l\geq 1$}
\end{cases}
 \end{equation}
 where $\Tr$ is the normalized trace map to $\mQ_p$.
Since $$C^{\pm}_{n,j}\cdot \sum_{a\text{ mod }p^n}\chi(a)\otimes\zeta_{p^n}^a=\left(\sum_{\chi}C^{\pm}_{n,j,\phi}\cdot e_{\phi}\right)\cdot e_{\chi}$$ (here $\chi$ is the character of conductor $p^n$ in the statement of Corollary \ref{5.3}), we have \begin{align*}
 &\bigg(\Id\otimes \Tr \bigg)\left(C^{\pm}_{n,j}\cdot \sum_{a\text{ mod }p^n}\chi(a)\otimes\zeta_{p^n}^a\right)\\
 =& \bigg(\Id\otimes \Tr \bigg)\left(C_{n,j,\chi^{-1}}^{\pm}\cdot e_{\chi^{-1}}\cdot e_{\chi}\right)\\
 =& C_{n,j,\chi^{-1}}^{\pm}\cdot \bigg(\Id\otimes \Tr \bigg)\left(\sum_{a\text{, }b\in\left(\mZ/p^n\right)^*}\chi(ab^{-1})\otimes\zeta_{p^n}^{a+b}\right)\\
 =& C_{n,j,\chi^{-1}}^{\pm}\cdot \bigg(\Id\otimes \Tr \bigg)\left(\sum_{c\in\left(\mZ/p^n\right)^*}\chi(c)\otimes\sum_{b\in\left(\mZ/p^n\right)^*}\zeta_{p^n}^{bc+b}\right)\\
 =& p^{n-1}(p-1)\cdot\chi(-1)\cdot C_{n,j,\chi^{-1}}^{\pm}.
 \end{align*}
Combining this with Corollary \ref{5.7} gives equation \eqref{5.8}.
\end{proof}

\vskip 3em

\section{Images of \texorpdfstring{$\z_M^{\pm}(f)$}{z_M} under dual exponential maps, supercuspidal case}

\noindent In this section, we continue assuming that the local ($p$-adic) smooth representation of $\GL_2\left(\mQ_p\right)$ attached to the modular form $f$ is supercuspidal.

Before we state our main theorem, let's recall the definition of integrating classes in the local Iwasawa cohomology.

For any class $c\in H^1\big(\mathbb{Q}_p, V_f^*\otimes_{\mathbb{Z}_p}\Lambda\big)$ where $\Lambda$ is the Iwasawa algebra viewed as the space of measures on $\mathbb{Z}_p^*$:

\begin{itemize}
\item If $\phi$ is a locally constant function on $\mathbb{Z}_p^*$ which factors through $\left(\mathbb{Z}/p^n\right)^*$, we can define 
\begin{equation}
 \int_{\mathbb{Z}_p^*}\phi\cdot c :=\left\{ \gamma\in G_{\mathbb{Q}_p(\zeta_{p^n})} \mapsto \int_{\mathbb{Z}_p^*}\phi(x)\cdot c(\gamma)\in V_f^* \right\}\in H^1\left(\mathbb{Q}_p(\zeta_{p^n}), V_f^*\right). 
\end{equation}

It is easy to check that the above map from $H^1\big(\mathbb{Q}_p, V_f^*\otimes_{\mathbb{Z}_p}\Lambda\big)$ to $H^1\big(\mathbb{Q}_p(\zeta_{p^n}), V_f^*\big)$ is well defined for any locally constant function $\phi$ which factors through $\left(\mathbb{Z}/p^n\right)^*$.

\item  If $\chi$ is a continuous character of $\mathbb{Z}_p^*$, then we can define the integration which takes value in $H^1\big(\mathbb{Q}_p, V_f^*(\chi)\big)$:
\begin{equation}
 \int_{\mathbb{Z}_p^*}\chi\cdot c :=\left\{ \gamma\in G_{\mathbb{Q}_p} \mapsto \int_{\mathbb{Z}_p^*}\chi(x)\cdot c(\gamma)\in V_f^* \right\}\in H^1\left(\mathbb{Q}_p, V_f^*(\chi)\right). 
\end{equation}

It is easy to check that the above map from $H^1\big(\mathbb{Q}_p, V_f^*\otimes_{\mathbb{Z}_p}\Lambda\big)$ to $H^1\big(\mathbb{Q}_p, V_f^*(\chi)\big)$ is well defined for any continuous character $\chi$ of $\mathbb{Z}_p^*$.

\item The above two definitions are compatible for finite order characters $\phi$ of $\mathbb{Z}_p^*$ with conductor $p^n$ in the following sense:

There is the following commutative diagram:
\begin{center}
\begin{tikzcd}
H^1\left(\mathbb{Q}_p, V_f^*\otimes_{\mathbb{Z}_p}\Lambda\right)\ar[rrrrr, "\text{First definition of }\int_{\mathbb{Z}_p^*}\phi"]\ar[drrrrr, "\text{Second definition of }\int_{\mathbb{Z}_p^*}\phi"'] &&&&&  H^1\left(\mathbb{Q}_p(\zeta_{p^n}), V_f^*\right) \\
&&&&& H^1\left(\mathbb{Q}_p, V_f^*(\chi)\right) \ar[u, "\res"']
\end{tikzcd}
\end{center}
In the rest of this paper, for notational convenience, we will not mention which of the above definitions we are using. 
\end{itemize}

Recall that Shapiro's lemma gives an isomorphism between $H^1\big(\mathbb{Q}_p, V_f^*\otimes_{\mathbb{Z}_p}\Lambda\big)$ and $\HH^1_{\Iw}\big(\mathbb{Q}_p, V_f^*\big)$. We can regard the above integrations as defined on $\HH^1_{\Iw}\big(\mathbb{Q}_p, V_f^*\big)$.

\begin{lemma}\label{lemma 6.1} We denote $c_{j,n}$ the projection from $\HH_{\Iw}^1(\mathbb{Q}_p,V_f^*)$ to $H^1(\mathbb{Q}_p(\zeta_{p^n}),V_f^*(-j))$. Then for any $\z\in\HH^1_{\Iw}(\mathbb{Q}_p, V_f^*)$ and any $a\in\left(\mathbb{Z}/p^n\right)^*$:
\begin{equation}\label{equi-5.3}
    \int_{a+p^n\mathbb{Z}_p}x^{-j}\cdot \z=\sigma_a\left(c_{j,n}(\z)\right).
\end{equation}
Therefore, for any locally constant function $\phi$ of conductor $p^n$, \begin{equation}
 \int_{\mathbb{Z}_p^*}\phi(x)x^{-j}\cdot\z=\sum_{a\text{ mod }p^n}\phi(a)\cdot \sigma_a\left(c_{j,n}(\z)\right). 
 \end{equation}
\end{lemma}

\begin{proof} We first recall the explicit description of Shapiro's lemma. 

Let $G$ be a profinite group, and $H\leq G$ be a subgroup of finite index such that $G/H$ is commutative, $V$ a representation of the group $G$ with coefficients in $L$. The isomorphism given by Shapiro's lemma $$S:\text{\quad} H^1\left(G\text{, }\Hom_{L[H]}(L[G]\text{, }V)\right)\xrightarrow{\cong}H^1\left(H\text{, }V\right)$$ 
has formula $S(c)=\left\{h\in H\mapsto c(h)\left([1]\right)\right\}\in H^1\left(H\text{, }V\right)$, where ``$1$" is the identity element in $G$. Here, the action of $G$ on $\Hom_{L[H]}(L[G]\text{, }V)$ is given by formula $(g.F)([g'])=F[g'g]$.

Since $V$ a representation of the group $G$, there is an isomorphism of $G$-representations $$\alpha:\text{\quad} \Hom_{L[H]}(L[G]\text{, }V)\xrightarrow{\cong} V\otimes_L L\left[G/H\right]$$ given by formula $$\alpha(f)=\sum_{\overline{z}\in G/H}\left(z^{-1}\cdot f([z])\right)\otimes[\overline{z}^{-1}],$$ where $z$ is any lift of $\overline{z}$ in $G$. Here, the action of $G$ on $\Hom_{L[H]}(L[G]\text{, }V)$ is still as described as before, and the action of $G$ on $V\otimes_L L\left[G/H\right]$ is given by formula $g.(v\otimes [g'])=(g.v)\otimes[gg']$.

Thus the isomorphism $\alpha$ induces an isomorphism $$\alpha_*:\text{\quad} H^1\left(G,\Hom_{L[H]}(L[G]\text{, }V)\right)\xrightarrow{\cong} H^1\left(G\text{, }V\otimes_L L\left[G/H\right]\right),$$ so we have the composite: $$\alpha_* \circ S^{-1}: \text{\quad} H^1\left(H\text{, }V\right)\xrightarrow{\cong} H^1\left(G\text{, }V\otimes_L L\left[G/H\right]\right).$$

If $H'$ is a subgroup of $H$ such that $G/H'$ is commutative, then we have the following commutative diagram:

\begin{center}
\begin{tikzcd}
H^1\left(H'\text{, }V\right)\ar[d, "\cores"]\ar[rr,"\cong"', "\alpha_* \circ S^{-1}"]  &&  H^1\left(G\text{, }V\otimes_L L\left[G/H'\right]\right)\ar[d, "\Proj"]\\
H^1\left(H\text{, }V\right) \ar[rr,"\cong"', "\alpha_* \circ S^{-1}"] &&  H^1\left(G\text{, }V\otimes_L L\left[G/H\right]\right)
\end{tikzcd}
\end{center}
where the vertical map on the right is induced by the natural projection from $G/H'$ to $G/H$.

The inverse limit of the map $\alpha_*\circ S^{-1}$ with respect to corestrictions and projections gives the isomorphism between $\HH^1_{\Iw}\left(\mQ_p, V\right)$ and $H^1\left(\mQ_p, V\otimes_{\mZ_p} \Lambda \right)$, which we will denote by $\sS: \HH^1_{\Iw}\left(\mQ_p, V\right)\xrightarrow{\cong}H^1\left(\mQ_p, V\otimes_{\mZ_p} \Lambda \right)$.

\vskip 0.5em

To prove equation \eqref{equi-5.3} for $j=0$, we let $G=\Gal\left(\overline{\mQ}_p/\mQ_p\right)$ and $H=\Gal\left(\overline{\mQ}_p/\mQ_p\left(\zeta_{p^n}\right)\right)$. Notice that the following composite 
\begin{center}
\begin{tikzcd}
H^1\left(G\text{, }\Hom_{L[H]}(L[G], V)\right)\ar[r,"\alpha_*","\cong"'] & H^1\left(G\text{, }V\otimes_L L\left[G/H\right]\right)\ar[rr, "\int_{a+p^n\mZ_p}1"]  && H^1\left(H\text{, }V\right)
\end{tikzcd}
\end{center}
sends any class $c\in H^1\left(G\text{, }\Hom_{L[H]}(L[G], V)\right)$ to $$\left\{h\in H \mapsto \sigma_a\cdot \left(c(h)([\sigma_a^{-1}])\right)\right\}=\sigma_a\cdot \left\{h\in H \mapsto  c(h)([1]) \right\}=\sigma_a\cdot S(c)\text{ }\in H^1\left(H\text{, }V\right)$$ using the explicit formula of the map $\alpha_*$ above. Here, the action of $\sigma_a$ on $H^1\left(H\text{, }V\right)$ is through the usual action of $G/H$ on $H^1\left(H\text{, }V\right)$. Thus we have $$\sigma_a\left(c_{0,n}(\z)\right)=\int_{a+p^n\mZ_p}1\cdot (\alpha_*\circ S^{-1} \left(c_{0,n}(\z)\right))=\int_{a+p^n\mZ_p}1\cdot \z.$$

We can then deduce equation \eqref{equi-5.3} for general $j$ by using the following commutative diagram:
\begin{center}
\begin{tikzcd}
H^1\left(\mQ_p, V\otimes_{\mZ_p} \Lambda \right)\ar[d, "\cong"] & \HH^1_{\Iw}\left(\mQ_p, V\right)\ar[d, "\cong"]\ar[rrd, "\int_{\mZ_p^*}x^j\phi(x)"]\ar[l,"\sS"',"\cong"] \\
H^1\left(\mQ_p, V(j)\otimes_{\mZ_p} \Lambda \right) &  \HH^1_{\Iw}\left(\mQ_p, V(j)\right)\ar[rr, "\int_{\mZ_p^*}\phi(x)"']\ar[l,"\sS"',"\cong"] &&  H^1\left(\mQ_p\left(\zeta_{p^n}, V(j)\right)\right) \\
\end{tikzcd}
\end{center}
where the vertical isomorphism on the left is induced by sending $v\otimes [\sigma_a]\in V\otimes_{\mZ_p} \Lambda $ to $v\otimes a^j \otimes [\sigma_a]\in V(j)\otimes_{\mZ_p} \Lambda $, and the locally constant function $\phi$ factors through $\left(\mZ/p^n\right)^*$.

\end{proof}

\vskip 1em

The follwing lemma will be useful later on:

\begin{lemma}\label{lem-5.2}
Assume $F$ is a finite extension of $\mathbb{Q}_p$. If $V$ is a $p$-adic continuous representation of $G_F$, and $K/F$ is a finite extension, then following diagrams are commutative:
\begin{center}
\begin{tikzcd}
H^1\left(K,V\right)\ar[d,"\cores"]\ar[r,"\exp^*"] & \D_{\dR,K}(V)\ar[d,"\Tr^K_F"] & & H^1\left(K,V\right)\ar[r,"\exp^*"] & \D_{\dR,K}(V) \\
H^1\left(F, V\right)\ar[r,"\exp^*"] & \D_{\dR,F}(V) & & H^1\left(F, V\right)\ar[u, "\res"]\ar[r,"\exp^*"] & \D_{\dR,F}(V) \ar[u, hook, "\text{ natural inclusion}"']
\end{tikzcd}
\end{center}
Where the trace map $\Tr^K_F$ appeared above is the unnormalized one.
\end{lemma}

\begin{proof} Recall the definition of $\exp^*$ is as the following commutative diagram: 
\begin{center}
\begin{tikzcd}
H^0\left(F, V\otimes\B_{\dR}\right)\ar[rr, "\smile\log(\cycl)", "\cong"']     && H^1\left(F, V\otimes\B_{\dR}\right) \\
&&   H^1\left(F, V\right)\ar[u]\ar[ull, "\exp^*" ]
\end{tikzcd}
\end{center}
where $\log(\cycl)\in H^1\left(F,\mQ_p\right)$.

Therefore, Lemma \ref{lem-5.2} follows from the following two commutative diagrams for cup products, when $F\subset K$: 
\begin{center}
\begin{tikzcd}
H^0\left(F, V\otimes\B_{\dR}\right)\ar[r, phantom, "\times"]\ar[d,"\res"]     &   H^1\left(F, \mQ_p\right) \ar[r,"\smile"]\ar[d,"\res"]  & H^1\left(F, V\otimes\B_{\dR}\right)\ar[d,"\res"] \\
H^0\left(K, V\otimes\B_{\dR}\right)\ar[r, phantom, "\times"]     &   H^1\left(K, \mQ_p\right) \ar[r,"\smile"]  & H^1\left(K, V\otimes\B_{\dR}\right)
\end{tikzcd}
\end{center}

\begin{center}
\begin{tikzcd}
H^0\left(F, V\otimes\B_{\dR}\right)\ar[r, phantom, "\times"]     &   H^1\left(F, \mQ_p\right) \ar[r,"\smile"]\ar[d,"\res"]  & H^1\left(F, V\otimes\B_{\dR}\right) \\
H^0\left(K, V\otimes\B_{\dR}\right)\ar[r, phantom, "\times"]\ar[u,"\cores"]     &   H^1\left(K, \mQ_p\right) \ar[r,"\smile"]  & H^1\left(K, V\otimes\B_{\dR}\right)\ar[u,"\cores"]
\end{tikzcd}
\end{center}

\end{proof}

\vskip 1em

We can now state our main theorem in the case when $\pi_p^{\sm}(f)$ is supercuspidal:

\vskip 1em

\begin{theorem}\label{thm-4.3}
The element $\z_M^{\pm}(f)\in\HH_{\Iw}^1\big(\mathbb{Q}_p,V_f^*\big)$ has the following property: 

For any $0\leq j\leq k-2$, any locally constant character $\chi$ of $\mathbb{Z}_p^*$ with conductor $p^n$ $(n\geq 0)$, \begin{equation}\label{main equation 2}
    \exp^*\left(\int_{\mathbb{Z}_p^*}\chi(x)x^{-j}\cdot\z_M^{\pm}(f)\right)=\pm \frac{1}{\tau(\chi)}\cdot \frac{\widetilde{\Lambda}(f, \chi  ,j+1)}{j!}\cdot \overline{f}_{\chi}\cdot t^j,
\end{equation}
where $$\overline{f}_{\chi}\cdot t^j:=\left(\sum_{a}\chi(a)\overline{f}\otimes\zeta_{p^n}^a\right)\cdot t^j=\overline{f}\cdot e_{\chi}\cdot t^j\in \Fil^0\D_{\dR}(V_f^*(\chi)(-j)).$$ 

Here, the sign $\pm$ for $\z_M^{\pm}(f)$ is chosen such that $\chi(-1)=\pm(-1)^j$. If the sign doesn't match, then the LHS of equation (\ref{main equation 2}) equals 0.
\end{theorem}

\begin{proof}
We denote $c_{j,n}$ the projection from $\HH_{\Iw}^1\big(\mathbb{Q}_p,V_f^*\big)$ to $\HH^1(\mathbb{Q}_p(\zeta_{p^n}),V_f^*(-j))$.

According to \cite{C2} Lemma VIII.2.1, $\exp^*\left(c_{j,n}(\z_{M}^{\pm}(f))\right)$ equals $\frac{1}{p^n}$ times the image of $\iota_m\left(\Exp^*(\z_M^{\pm}(f))\right)$ modulo $\gamma_n-\varepsilon(\gamma_n)^{j}$ for all $m$ sufficiently large, where $\varepsilon$ is the cyclotomic character.

In other words, $\exp^*\left(c_{j,n}(\z_{M}^{\pm}(f))\right)$ equals $\frac{1}{p^n}$ times $\Tr_{L_n}$ applied to the coefficient of $t^j$ of $\iota_m\left(\Exp^*(\z_M^{\pm}(f))\right)$, where $\Tr_{L_n}$ is the normalized trace map to $L_n$. Since $$\iota_m\circ\Exp^*(\z^{\pm}_M)=\overline{f}\otimes\sum_{j=0}^{k-2}C_{m,j}^{\pm}t^j+\text{ ``Something in $\D_{\dR}(V_f^*)\otimes_{L}t^{k-1}L_m[[t]]$",}$$ we have \begin{equation}
 \exp^*\left(c_{j,n}(\z_{M}^{\pm}(f))\right)=\frac{1}{p^n}\cdot C_{n,j}^{\pm}\cdot \overline{f}\cdot t^j \text{\quad} \in \D_{\dR}(V_f^*(-j))\otimes_{\mathbb{Q}_p}\mathbb{Q}_p(\zeta_{p^n}).
 \end{equation}
 
 \vskip 0.5em
 
 Using Lemma \ref{lemma 6.1}, we conclude that \begin{align*}
 \exp^*\left(\int_{\mathbb{Z}_p^*}\chi(x)x^{-j}\cdot\z_{M}^{\pm}(f)\right) &= \frac{1}{p^n}\sum_{a\text{ mod }p^n}\chi(a)\cdot \sigma_a.\exp^*\left(c_{j,n}(\z_{M}^{\pm}(f))\right)\\
 &=\frac{1}{p^n}\sum_{a\text{ mod }p^n}\chi(a)\cdot \sigma_a\left(C_{n,j}^{\pm}\right)\cdot \overline{f}\cdot t^j\\
 &=\frac{1}{p^n}\sum_{a\text{ mod }p^n}\sum_{\chi'}C_{n,j,\chi'}^{\pm}\cdot\chi(a)\chi'(a)^{-1}\cdot\overline{f}\cdot e_{\chi'} \cdot t^j\\
 &= \frac{p-1}{p}\cdot C_{n,j,\chi}^{\pm}\cdot\overline{f}\cdot e_{\chi} \cdot t^j\\
 &\text{\quad(Using equation \eqref{5.8})}\\
 &= \pm \frac{1}{\tau(\chi)}\cdot \frac{\widetilde{\Lambda}(f,\chi,j+1)}{j!}\cdot\overline{f}_{\chi} \cdot t^j.
 \end{align*}

This finishes the proof. 
 \end{proof}

\begin{remark}
Since we are working with the case when the modular form $f$ is supercuspidal at $p$, we know $f$ necessarily has bad reduction at $p$. Hence $\widetilde{\Lambda}(f,\chi,j+1)=\widetilde{\Lambda}_{(p)}(f,\chi,j+1)$ for all $\chi$ and $j$. Therefore, the formula in Theorem \ref{thm-4.3} can also be written as \begin{equation}
 \exp^*\left(\int_{\mathbb{Z}_p^*}\chi(x)x^{-j}\cdot\z_M^{\pm}(f)\right)=\pm \frac{1}{\tau(\chi)}\cdot \frac{\widetilde{\Lambda}_{(p)}(f, \chi  ,j+1)}{j!}\cdot \overline{f}_{\chi}\cdot t^j.
 \end{equation}

\end{remark}

\vskip 3em 

\section{Explicit \texorpdfstring{$p$}{p}-adic Local Langlands correspondence, principal series case}

\noindent The main purpose of this section is to compare the notations used in \cite{BB}, \cite{BB1}, \cite{C2} with the one used in this paper, and then describe explicitly the isomorphism (as representations of $B(\mQ_p)$)
\begin{equation}
   \sF:\text{ }\Pi\left(V_f\right)^*\xrightarrow{\cong} \Big(\lim\limits_{\substack{\longleftarrow \\ \psi}}\D(V_f^*)\Big)_{\bb}
\end{equation}
in the cases when $\pi_p^{\sm}(f)\cong \ind_{\overline{B}(\mQ_p)}^{\GL_2(\mQ_p)}\left(\unr(\alpha)\otimes\unr(\beta p^{-1})\right)$ with $|\alpha\beta|=|p^{k-1}|$, $\alpha/\beta\neq p^{\pm 1}$ and $0<v_p(\alpha),v_p(\beta)<k-1$. This is equivalent to the condition that $V_f|_{G_{\mQ_p}}$ is crystaline and absolutely irreducible. Here, $\ind$ means the smooth induction, $\overline{B}(\mQ_p)$ is the group of lower triangular matrices in $\GL_2(\mQ_p)$, and $\unr(\lambda)$ is the unramified character of $\mQ_p^*$ that sends $p\in\mQ_p^*$ to $\lambda$. 

The main reference for this section is \cite{C2} (Section II.3.3), \cite{BB} and \cite{BB1}.

\vskip 1em
Recall that the theory of intertwining operators gives an isomorphism \begin{equation}
 I:\text{ }\ind_{\overline{B}(\mQ_p)}^{\GL_2(\mQ_p)}\left(\unr(\alpha)\otimes\unr(\beta p^{-1})\right)\xrightarrow{\cong} \ind_{\overline{B}(\mQ_p)}^{\GL_2(\mQ_p)}\left(\unr(\beta)\otimes\unr(\alpha p^{-1})\right).
 \end{equation}
 
 We define $$\pi^{\la}(\alpha):=\Ind_{\overline{B}(\mQ_p)}^{\GL_2(\mQ_p)}\left(\unr(\alpha)\otimes\unr(\beta p^{-1})x^{2-k}\right)$$ and $$\pi^{\la}(\beta):=\Ind_{\overline{B}(\mQ_p)}^{\GL_2(\mQ_p)}\left(\unr(\beta)\otimes\unr(\alpha p^{-1})x^{2-k}\right),$$
 where $\Ind$ means the locally analytic induction. 
 
 It can be seen easily that $\pi^{\la}(\alpha)$ contains $\Pi(V_f)^{\lalg}:=\left(\Sym^{k-2}\right)^{\vee}\otimes\pi_p^{\sm}(f)$ as a subrepresentation. Using the intertwining operator, we know $\pi^{\la}(\beta)$ also contains $\Pi(V_f)^{\lalg}$. We thus have an inclusion (see Corollary 7.2.5 of \cite{BB}): 
 \begin{equation}\label{6.3}
 \sJ_{\text{E}}:\text{ }\pi^{\la}(\alpha)\oplus_{\Pi(V_f)^{\lalg}}\pi^{\la}(\beta)\hookrightarrow \Pi(V_f),
 \end{equation}
and in fact, Liu has proved in \cite{L} that the above map identifies the source with the space of locally analytic vectors of $\Pi(V_f)$.

\vskip 1em

Let $\LA_c\left(\mQ_p,L\right)$ be the space of locally analytic functions with compact support on $\mQ_p$ taking values in $L$, where $L$ is the coefficient of the representation $V_f$. 

We define a map $$i_{\alpha}:\text{ } \LA_c\left(\mQ_p,L\right)\hookrightarrow \pi^{\la}(\alpha)$$ 
by sending every function $h\in \LA_c\left(\mQ_p,L\right)$ to a function $i_{\alpha}(h)$ defined on $\GL_2\left(\mQ_p\right)$, such that $$i_{\beta}(h)\left(\begin{pmatrix}1 & x \\ 0 & 1\end{pmatrix}\right)=h(x).$$

We define $i_{\beta}:\text{ } \LA_c\left(\mQ_p,L\right)\hookrightarrow \pi^{\la}(\beta)$ in the same way.

\vskip 1em

Recall that the induction from the lower triangular Borel $\overline{B}(\mQ_p)$ and the induction from the upper triangular Borel $B(\mQ_p)$ are related by the following isomorphism:
\begin{equation}
W_{\alpha}:\text{ }\Ind_{\overline{B}(\mQ_p)}^{\GL_2(\mQ_p)}\left(\unr(\alpha)\otimes\unr(\beta p^{-1})x^{2-k}\right)\xrightarrow{\cong}\Ind_{B(\mQ_p)}^{\GL_2(\mQ_p)}\left(\unr(\beta p^{-1})x^{2-k}\otimes\unr(\alpha)\right),
\end{equation}
where for every $h\in \Ind_{\overline{B}(\mQ_p)}^{\GL_2(\mQ_p)}\left(\unr(\alpha)\otimes\unr(\beta p^{-1})x^{2-k}\right)$, $W_{\alpha}(h)$ is a function on $\GL_2(\mQ_p)$ such that for every $g\in\GL_2\left(\mQ_p\right)$, $$W_{\alpha}(h)(g)=h\left(\begin{pmatrix}0 & 1 \\ 1 & 0\end{pmatrix}g\right).$$

\vskip 1em

We then define a map $$j_{\alpha}:\text{ } \LA_c\left(\mQ_p,L\right)\hookrightarrow \Ind_{B(\mQ_p)}^{\GL_2(\mQ_p)}\left(\unr(\beta p^{-1})x^{2-k}\otimes\unr(\alpha)\right)$$ 
by sending every function $h\in \LA\left(\mQ_p,L\right)$ to a function $j_{\alpha}(h)$ defined on $\GL_2\left(\mQ_p\right)$, such that $$j_{\beta}(h)\left(\begin{pmatrix}0 & 1 \\ 1 & x\end{pmatrix}\right)=h(x).$$

We then have the following commutative diagram:

\begin{center}
\begin{tikzcd}
\LA_c\left(\mQ_p,L\right)\ar[rrr,hook, "i_{\alpha}"]\ar[rrrd,hook, "j_{\alpha}"] &&& \Ind_{\overline{B}(\mQ_p)}^{\GL_2(\mQ_p)}\left(\unr(\alpha)\otimes\unr(\beta p^{-1})x^{2-k}\right)  \ar[d, "W_{\alpha}", "\cong"']\\
&&& \Ind_{B(\mQ_p)}^{\GL_2(\mQ_p)}\left(\unr(\beta p^{-1})x^{2-k}\otimes\unr(\alpha)\right)
\end{tikzcd}
\end{center}

\vskip 1em

All the above definitions and discussions hold when we replace $\alpha$ by $\beta$.

\vskip 1em

Notice that the central character of $\Pi\left(V_f\right)$ is $\delta(z)=z^{2-k}\left(\unr(\frac{\alpha\beta}{p})(z)\right)$, which is a unitary character; while the determinant of $V_f|_{G_{\mQ_p}}$ is $(z|z|)^{-1}\delta(z)$ for $z\in\mQ_p^*$, where we are identifying $\mQ_P^*$ with the Weil group $W_{\mQ_p}$ by sending $p\in\mQ_p^*$ to the inverse of Frobenius. Hence we have $V_f^*(1)\cong V_f\otimes z^2|z|^2\delta^{-1}(z)$ and $\Pi\left(V_f^*(-1)\right)\cong\Pi\left(V_f\right)\otimes\left(\delta^{-1}\circ\det\right)$.

\vskip 1em

In \cite{BB}, Berger and Breuil defined $$LA(\alpha):=\Ind_{B(\mQ_p)}^{\GL_2(\mQ_p)}\left(\unr(\alpha^{-1})\otimes x^{k-2}\unr( p\beta^{-1})\right)$$ and $$LA(\beta):=\Ind_{B(\mQ_p)}^{\GL_2(\mQ_p)}\left(\unr(\beta^{-1})\otimes x^{k-2}\unr(p\alpha^{-1})\right)$$ 

Similar to the inclusion (\ref{6.3}), we have \begin{equation}\sJ_{\text{BB}}:\text{ }LA(\alpha)\oplus_{\Pi(V_f^*(-1))^{\lalg}}LA(\beta)\hookrightarrow \Pi(V_f^*(-1)),\end{equation} see Corollary 7.2.5 of \cite{BB}. (There is a difference of notations for Banach space representations between \cite{BB} and this paper, which will be explained in Remark \ref{rem-not} later in this section.)

We now have the following commutative diagram, and we denote the map of the dashed arrow by $s_{\alpha}$: 

\begin{center}
\begin{tikzcd}
\LA_c\left(\mQ_p,L\right)\ar[rrrdd, dashed, bend right=15, "s_{\alpha}"']\ar[rrr,hook, "i_{\alpha}"]\ar[rrrd,hook, "j_{\alpha}"] &&& \pi^{\la}(\alpha)  \ar[d, "W_{\alpha}", "\cong"']\ar[rr, hook, "\sJ_{\text{E}}"]&& \Pi(V_f)\ar[dd, dotted, "\vartheta", "\sim"']\\
&&& \Ind_{B(\mQ_p)}^{\GL_2(\mQ_p)}\left(\unr(\beta p^{-1})x^{2-k}\otimes\unr(\alpha)\right)\ar[d,"\text{multiplied by }\delta^{-1}\circ\det"]  \\
&&& LA(\alpha)\ar[rr, hook, "\sJ_{\text{BB}}"']  && \Pi(V_f^*(-1))
\end{tikzcd}
\end{center}
Here, we write ``multiplied by $\delta^{-1}\circ\det$" (which is not a $\GL_2(\mQ_p)$-equivariant map) because we are viewing both the source and target of this map as some space of functions on the group $\GL_2(\mQ_p)$. 

We have similar diagram when we replace $\alpha$ by $\beta$, and the dotted arrow $\vartheta$ from $\Pi(V_f)$ to $\Pi(V_f^*(-1))$ is induced by the two diagrams for  $\alpha$ and $\beta$.

\begin{remark}\label{rem-def}
The map $\vartheta: \Pi(V_f)\rightarrow\Pi(V_f^*(-1))$ is just an isomorphism of topological vector spaces. It is not $\GL_2(\mQ_p)$-equivariant.
\end{remark}

We have the following lemma, whose proof is just a careful tracking of the diagram above:

\begin{lemma}\label{lem-def}
For any $z\in\left(\Pi(V_f)\right)^*$, if we define $\nu_{\alpha}(z):=i_{\alpha}^*\circ\sJ_{\text{\normalfont{E}}}^*(z)\in\sD(\mQ_p)$ and $\mu_{\alpha}(z):=s_{\alpha}^*\circ\sJ_{\text{\normalfont{BB}}}^*\circ\vartheta_*(z)\in\sD(\mQ_p)$, then $\nu_{\alpha}(f)$ and $\mu_{\alpha}(f)$ are related by the following formula: 
\begin{equation}
  \int_{\mQ_p}f(x)\nu_{\alpha}(z)=\int_{\mQ_p}f(-x)\mu_{\alpha}(z)
\end{equation} 
for any function $f(x)\in\LA_c\left(\mQ_p,L\right)$.

We have similar results when we replace $\alpha$ by $\beta$.
\end{lemma}
\begin{flushright}
$\Box$
\end{flushright}

\begin{remark}
The distribution $\mu_{\alpha}(z)$ is the distribution on $\mQ_p$ corresponding to $z$ defined in \cite{BB}.
\end{remark}

\vskip 1em

\begin{definition}\label{def-6.1}
 We equip $\lim\limits_{\substack{\longleftarrow \\ \psi}}\D(V_f^*)$ with an action of $B(\mQ_p)$ as follows: 
\begin{itemize}
\item The matrix $\begin{pmatrix} z & 0 \\ 0 & z\end{pmatrix}$ in the center acts by the scalar $\delta^{-1}(z)$ on each copy of $\D(V_f^*)$.
\item The matrix $\begin{pmatrix} p & 0 \\ 0 & 1\end{pmatrix}$ acts via $\psi^{-1}$, that is, shifting to the right.
\item The matrix $\begin{pmatrix} a & 0 \\ 0 & 1\end{pmatrix}$ where $a\in\mZ_p^*$ acts by $\sigma_a$ on each copy of $\D(V_f^*)$.
\item The matrix $\begin{pmatrix} 1 & b \\ 0 & 1\end{pmatrix}$ where $b\in\mQ_p$ acts via multiplying $(1+T)^b$ on the last copy of $\D(V_f^*)$.
\end{itemize}
\end{definition}

\begin{remark}\label{6.2} If we take $V=V_f^*$ and $\chi=\delta^{-1}$, then $\lim\limits_{\substack{\longleftarrow \\ \psi}}\D(V_f^*)$ with the action of $B\left(\mQ_p\right)$ as defined in \cite{BB1} (Definition 3.4.3) is isomorphic to $\lim\limits_{\substack{\longleftarrow \\ \psi}}\D(V_f^*)\otimes \left(\delta\circ\det\right)$ with the action of $B\left(\mQ_p\right)$ as defined above in Definition \ref{def-6.1}.
\end{remark}

\vskip 1em

In \cite{BB} and \cite{BB1}, Berger and Breuil proved the following theorem (See Theorem 8.1.1 in \cite{BB}, or Th\'eor\`eme 5.2.7 in \cite{BB1}), which we present here using our notation:

\begin{theorem}[Berger \& Breuil]\label{main-ps}
Assume that $\pi_p^{\sm}(f)\cong \ind_{\overline{B}(\mQ_p)}^{\GL_2(\mQ_p)}\left(\unr(\alpha)\otimes\unr(\beta p^{-1})\right)$ with $|\alpha\beta|=|p^{k-1}|$, $\alpha/\beta\neq p^{\pm 1}$ and $0<v_p(\alpha),v_p(\beta)<k-1$. Then under the $B(\mQ_p)$-equivariant isomorphism of topological vector spaces $($as defined by Colmez in \cite{C2}$)$
\begin{equation}\label{main-equiv-ps}
   \sF:\text{ }\Pi\left(V_f\right)^*\xrightarrow{\cong} \Big(\lim\limits_{\substack{\longleftarrow \\ \psi}}\D(V_f^*)\Big)_{\bb},
\end{equation}
we have the explicit formula \begin{equation}
 \sF(z)=\Big(\alpha^{-N}\mu_{\alpha,N}(z)\otimes e_{\alpha}+\beta^{-N}\mu_{\beta,N}(z)\otimes e_{\beta}\Big)_N.
 \end{equation}
Here, we are identifying $\Big(\lim\limits_{\substack{\longleftarrow \\ \psi}}\D(V_f^*)\Big)_{\bb}$ with $\lim\limits_{\substack{\longleftarrow \\ \psi}}\N(V_f^*)$ (see Proposition 8.1.2 in \cite{BB}), and we view $\N(V_f^*)$ as a subspace contained in $\B_{\rig,\mQ_p}^+\otimes\D_{\crys}(V_f^*)$ (see Proposition 5.5.3 in \cite{BB}). Also, $e_{\alpha}$ (resp. $e_{\beta}$) is a Frobenius eigenvector in $\D_{\crys}(V_f^*)$ with eigenvalue $\alpha^{-1}$ (resp. $\beta$) such that $\overline{f}=e_{\alpha}+e_{\beta}\in\Fil^0\D_{\dR}(V_f^*)$, and $\mu_{\alpha,N}(z)$, $\mu_{\beta,N}(z)$ are distributions on $\mZ_p$ (viewed as elements in $\B_{\rig,\mQ_p}^+$ via the Amice transform\footnote{Later in this and the next section, we will always identify $\sD\left(\mZ_p\right)$ with $\B_{\rig,\mQ_p}^+$ using the Amice transform without mentioning explicitly. This should not cause any confusion.}) defined as follows:

For any $z\in \Pi\left(V_f\right)^*$, let $\mu_{\alpha}(z)$ and $\mu_{\beta}(z)$ be as defined in Lemma \ref{lem-def}. We define for every $N\in\mZ_{\geq 0}$ two distributions $\mu_{\alpha,N}(z)$ and $\mu_{\beta,N}(z)$ on $\mZ_p$ using the formulas $$\int_{\mZ_p}f(x)\mu_{\alpha,N}(z):=\int_{\frac{1}{p^N}\mZ_p}f(p^Nx)\mu_{\alpha}(z)$$ and $$\int_{\mZ_p}f(x)\mu_{\beta,N}(z):=\int_{\frac{1}{p^N}\mZ_p}f(p^Nx)\mu_{\beta}(z)$$ for all $f(x)\in\LA_c\left(\mZ_p, L\right)$.
\end{theorem}
\begin{flushright}
$\Box$
\end{flushright}

\begin{remark}\label{rem-not}
The notations we are using here is related to the notations used in \cite{BB} and \cite{BB1} in the following way: for any continuous representation $V$ of $G_{\mQ_p}$, the $\Pi(V)$ in this paper is the $\Pi(V(1))$ in \cite{BB} and \cite{BB1}.

In \cite{BB} and \cite{BB1}, Theorem \ref{main-ps} is formulated as $$\Pi\left(V_f^*\right)^*\xrightarrow{\cong} \Big(\lim\limits_{\substack{\longleftarrow \\ \psi}}\D(V_f^*)\Big)_{\bb},$$ with their notation of  $\Pi\left(V_f^*\right)$ and their definition of $B\left(\mQ_p\right)$ acting on $ \Big(\lim\limits_{\substack{\longleftarrow \\ \psi}}\D(V_f^*)\Big)_{\bb}$. If we change both sides of the above isomorphism into our notations (see Lemma \ref{6.2}), we would get $$\Pi\left(V_f^*(-1)\right)^*\xrightarrow{\cong} \Big(\lim\limits_{\substack{\longleftarrow \\ \psi}}\D(V_f^*)\Big)_{\bb}\otimes \left(\delta\circ\det\right).$$

Notice that we have $\Pi\left(V_f^*(-1)\right)^*\otimes \left(\delta\circ\det\right)^{-1}\cong \left(\Pi\left(V_f^*(-1)\right)\otimes \left(\delta\circ\det\right)\right)^*$ and $\Pi\left(V_f^*(-1)\right)\cong\Pi\left(V_f\right)\otimes\left(\delta(z)\circ\det\right)^{-1}$, we see that our isomophism (\ref{main-equiv-ps}) coincides with the isomorphism introduced in \cite{BB} and \cite{BB1}.

The notation used in \cite{C2} is also compatible with the notation in \cite{BB} and \cite{BB1}, which has already been verified in \cite{L}.
\end{remark}

\vskip 3em 

\section{Images of \texorpdfstring{$\z_M^{\pm}(f)$}{z_M} under dual exponential maps, principal series case}

\noindent We continue using the same notations and assumptions as in the previous section. For example, $\pi_p^{\sm}(f)$ is still assumed to be a principal series. In \cite{E}, Emerton has proved the following result (Proposition 4.9 in \cite{E}): 

\begin{theorem}[Emerton, \cite{E}]\label{E}
With the notations as in Lemma \ref{lem-def}, we have $\nu_{\alpha}(\cM_f^{\pm})\Big|_{\mZ_p^*}$ (resp. $\nu_{\beta}(\cM_f^{\pm})\Big|_{\mZ_p^*}$) coincides with the $p$-adic $L$-function $L_{p,\alpha}(f)$ (resp. $L_{p,\beta}(f)$) as defined in \cite{MTT}.
\end{theorem}
\begin{flushright}
$\Box$
\end{flushright}

We can now prove our main theorem in the case when $\pi_p^{\sm}(f)$ is an unramified principal series:

\vskip 1em

\begin{theorem}\label{thm-4.3ps}
The element $\z_M^{\pm}(f)\in\HH_{\Iw}^1\left(\mathbb{Q}_p,V_f^*\right)$ has the following property: 

For any $0\leq j\leq k-2$, any locally constant character $\chi$ of $\mathbb{Z}_p^*$ with conductor $p^n$ $(n\geq 0)$, \begin{equation}\label{main equation 2ps}
    \exp^*\left(\int_{\mathbb{Z}_p^*}\chi(x)x^{-j}\cdot\z_M^{\pm}(f)\right)=\pm \frac{1}{\tau(\chi)}\cdot \frac{\widetilde{\Lambda}_{(p)}(f, \chi  ,j+1)}{j!}\cdot \overline{f}_{\chi}\cdot t^j\text{ }\in\text{ }\Fil^0\D_{\dR}(V_f^*(\chi)(-j)).
\end{equation}
Here, the sign $\pm$ for $\z_M^{\pm}(f)$ is chosen such that $\chi(-1)=\pm(-1)^j$. If the sign doesn't match, then the LHS of equation (\ref{main equation 2ps}) equals 0.
\end{theorem}

\begin{proof}
As in the proof of Theorem \ref{thm-4.3},  we again denote $c_{j,n}$ the projection from $\HH_{\Iw}^1(\mathbb{Q}_p,V_f^*)$ to $\HH^1(\mathbb{Q}_p(\zeta_{p^n}),V_f^*(-j))$, and we use Lemma VIII.2.1 in \cite{C2} that $\exp^*\left(c_{j,n}(\z_{M}^{\pm}(f))\right)$ equals $\frac{1}{p^n}$ times the image of $\iota_n\left(\Exp^*(\z_M^{\pm}(f))\right)$ modulo $\gamma_n-\varepsilon(\gamma_n)^{j}$, where $\varepsilon$ is the cyclotomic character.

\vskip 1em

We have the following commutative diagram: 
\begin{center}
\begin{tikzcd}
\D(V_f^*)^{\psi=1}\ar[rr, "\iota_n"]\ar[d, "F"] && L_n[[t]] \otimes_L \D_{\dR}(V_f^*) \\
 \left(\B_{\rig,\mQ_p}^+\otimes\D_{\crys}(V_f^*)\right)^{\psi=1}\ar[rru, "\iota_n\otimes \varphi^{-n}"'] 
\end{tikzcd}
\end{center}
where the vertical map $F$ is induced from the isomorphism $$\D_{\rig}^{+}(V_f^*)[\frac{1}{t}]\cong \B_{\rig,\mQ_p}^+[\frac{1}{t}]\otimes_{\mQ_p}\D_{\crys}(V_f^*) $$ as described in \cite{B}, together with the fact that $\D(V_f^*)^{\psi=1}\subset\N(V_f^*)\subset \B_{\rig,\mQ_p}^+\otimes\D_{\crys}(V_f^*)$ (see Theorem A.3 in \cite{B} and Proposition 5.5.3 in \cite{BB}, remember that we have always assumed $V_f|_{G_{\mQ_p}}$ to be absolutely irreducible). 

\vskip 1em

We write $$F\left(\Exp^*(\z_M^{\pm}(f))\right)=\Upsilon_{\alpha}\otimes e_{\alpha}+\Upsilon_{\beta}\otimes e_{\beta}\text{ }\in \B_{\rig,\mQ_p}^+\otimes\D_{\crys}(V_f^*).$$
Since $\Upsilon_{\alpha}\otimes e_{\alpha}+\Upsilon_{\beta}\otimes e_{\beta}$ is fixed by $\psi$, we have $\psi\left(\Upsilon_{\alpha}\right)=\alpha^{-1}\Upsilon_{\alpha}$ and $\psi\left(\Upsilon_{\beta}\right)=\beta^{-1}\Upsilon_{\beta}$. 

\vskip 1em

Notice that $\Exp^*(\z_M^{\pm}(f))=\fC\left(\cM_f^{\pm}\right)=\sF\left(\cM_f^{\pm}\right)$ is fixed by $\psi$, so if we write $$\sF\left(\cM_f^{\pm}\right)=\Big(\alpha^{-N}\mu_{\alpha,N}(\cM_f^{\pm})\otimes e_{\alpha}+\beta^{-N}\mu_{\beta,N}(\cM_f^{\pm})\otimes e_{\beta}\Big)_N,$$ then $$F\left(\Exp^*(\z_M^{\pm}(f))\right)=\mu_{\alpha,0}(\cM_f^{\pm})\otimes e_{\alpha}+\mu_{\beta,0}(\cM_f^{\pm})\otimes e_{\beta}=\mu_{\alpha}(\cM_f^{\pm})\Big|_{\mZ_p}\otimes e_{\alpha}+\mu_{\beta}(\cM_f^{\pm})\Big|_{\mZ_p}\otimes e_{\beta},$$hence $\Upsilon_{\alpha}=\mu_{\alpha}(\cM_f^{\pm})\Big|_{\mZ_p}$ and $\Upsilon_{\beta}=\mu_{\beta}(\cM_f^{\pm})\Big|_{\mZ_p}$.

\vskip 1em

Now since $\iota_n\left(\Upsilon_{\alpha}\right)=\Upsilon_{\alpha}\left(\zeta_{p^n}\exp\left(\frac{t}{p^n}\right)-1\right)= \int_{\mZ_p}\exp\left(\frac{tx}{p^n}\right)\otimes\zeta_{p^n}^x \cdot\mu_{\alpha}(\cM_f^{\pm})$,
we have \begin{align*}
 A_{j,n,\alpha}:=\text{Coefficient of $t^j$ of }\iota_n\left(\Upsilon_{\alpha}\right) &=\frac{1}{p^{jn}j!}\int_{\mZ_p}x^j\otimes\zeta_{p^n}^x\cdot\mu_{\alpha}(\cM_f^{\pm})   \\
 &=\frac{1}{p^{jn}j!}\sum_{i=0}^{+\infty}\int_{p^i\mZ_p^*}x^j\otimes\zeta_{p^n}^x\cdot\mu_{\alpha}(\cM_f^{\pm})\\
 &=\frac{1}{p^{jn}j!}\sum_{i=0}^{+\infty}p^i\int_{\mZ_p^*}x^j\otimes\zeta_{p^{n-i}}^x\cdot\psi^i\cdot\mu_{\alpha}(\cM_f^{\pm})\\
 &=\frac{1}{p^{jn}j!}\sum_{i=0}^{+\infty}\frac{p^i}{\alpha^i}\int_{\mZ_p^*}x^j\otimes\zeta_{p^{n-i}}^x\cdot\mu_{\alpha}(\cM_f^{\pm}).
 \end{align*}

Thus, for any locally constant character $\chi$ of $\mathbb{Z}_p^*$ with conductor $p^n$ ($n\geq 0$), \begin{align*}
   \alpha^n \sum_{a\in\left(\mZ/p^n\right)^*}\chi(a)\sigma_a\left(A_{j,n,\alpha}\right)&=\frac{\alpha^n}{p^{jn}j!}\sum_{i=0}^{+\infty}\frac{p^i}{\alpha^i}\int_{\mZ_p^*}x^j\left(\sum_{a\in\left(\mZ/p^n\right)^*}\chi(a)\otimes\zeta_{p^{n-i}}^{ax}\right)\cdot\mu_{\alpha}(\cM_f^{\pm}).
\end{align*}
It is easily seen that the Gauss sum $\sum_{a\in\left(\mZ/p^n\right)^*}\chi(a)\otimes\zeta_{p^{n-i}}^{ax}=0$ whenever $i>0$ because $\chi$ has conductor $p^n$. Therefore, if we pick the sign $\pm$ such that $\pm 1=(-1)^j\chi(-1)$, then we have
\begin{align*}
 \alpha^n \sum_{a\in\left(\mZ/p^n\right)^*}\chi(a)\sigma_a\left(A_{j,n,\alpha}\right)&=\frac{\alpha^n\sum_a\chi(a)\otimes\zeta_{p^n}^a}{p^{jn}j!}\int_{\mZ_p^*}x^j\chi^{-1}(x)\mu_{\alpha}(\cM_f^{\pm})\\
 &=(-1)^j\chi(-1)\frac{\alpha^n\sum_a\chi(a)\otimes\zeta_{p^n}^a}{p^{jn}j!}\int_{\mZ_p^*}x^j\chi^{-1}(x)\nu_{\alpha}(\cM_f^{\pm})\text{\quad (Lemma \ref{lem-def})}\\
 &=\pm\frac{\alpha^n\sum_a\chi(a)\otimes\zeta_{p^n}^a}{p^{jn}j!}\int_{\mZ_p^*}x^j\chi^{-1}(x)\nu_{\alpha}(\cM_f^{\pm})\\
 &=\pm \frac{\alpha^n\sum_a\chi(a)\otimes\zeta_{p^n}^a}{p^{jn}j!}\cdot L_{p,\alpha}(f)\left(x^j\chi^{-1}(x)\right)\text{\quad (Theorem \ref{E})}\\
 &= \pm \frac{\alpha^n\sum_a\chi(a)\otimes\zeta_{p^n}^a}{p^{jn}j!}\cdot j!\cdot p^{n(j+1)}\cdot \alpha^{-n}\cdot\tau(\chi)^{-1}\cdot (2\pi i)^{-(j+1)} \cdot \frac{L_{(p)}(f,\chi, j+1)}{\Omega_f^{\pm}}\\
 &\text{\quad\quad\quad\quad\quad\quad\qquad (See for example, \cite{K}, Page 269)}\\
 &=\pm p^n\cdot\frac{1}{\tau(\chi)}\cdot\frac{\widetilde{\Lambda}_{(p)}\left(f,\chi,j+1\right)}{j!}\cdot \sum_a\chi(a)\otimes\zeta_{p^n}^a.
 \end{align*}

Similarly, $\beta^n \sum_{a\in\left(\mZ/p^n\right)^*}\chi(a)\sigma_a\left(A_{j,n,\beta}\right)=\pm p^n\cdot\frac{1}{\tau(\chi)}\cdot\frac{\widetilde{\Lambda}_{(p)}\left(f,\chi,j+1\right)}{j!}\cdot \sum_a\chi(a)\otimes\zeta_{p^n}^a$, from which we can see that the RHS doesn't depend on $\alpha$ or $\beta$.
\vskip 1em

Hence, \begin{align*}
\exp^*\left(\int_{\mathbb{Z}_p^*}\chi(x)x^{-j}\cdot\z_M^{\pm}(f)\right)&= \sum_{a\in\left(\mZ/p^n\right)^*}\chi(a)\exp^*\left(c_{j,n}(\z_{M}^{\pm}(f))\right)\\
&=\frac{1}{p^n}\cdot \sum_{a\in\left(\mZ/p^n\right)^*}\chi(a)\sigma_a\left(\text{Coefficient of }t^j\text{ of }\iota_n\left(\Exp^*(\z_M^{\pm}(f))\right)\right)\cdot t^j\\
 &= \frac{1}{p^n}\cdot \sum_{a\in\left(\mZ/p^n\right)^*}\chi(a)\sigma_a\left(A_{j,n,\alpha}\otimes\alpha^ne_{\alpha}+A_{j,n,\beta}\otimes\beta^ne_{\beta}\right)\cdot t^j\\
 &=\pm \frac{1}{\tau(\chi)}\cdot\frac{\widetilde{\Lambda}_{(p)}\left(f,\chi,j+1\right)}{j!}\cdot\left(\sum_a\chi(a)\otimes\zeta_{p^n}^a\right)\otimes\left(e_{\alpha}+e_{\beta}\right)\cdot t^j\\
 &= \pm \frac{1}{\tau(\chi)}\cdot\frac{\widetilde{\Lambda}_{(p)}\left(f,\chi,j+1\right)}{j!}\cdot\left(\sum_a\chi(a)\otimes\zeta_{p^n}^a\right)\otimes \overline{f}\cdot t^j\\
 &=  \pm \frac{1}{\tau(\chi)}\cdot\frac{\widetilde{\Lambda}_{(p)}\left(f,\chi,j+1\right)}{j!}\cdot \overline{f}_{\chi}\cdot t^j\text{ \quad }\in\text{ }\Fil^0\D_{\dR}(V_f^*(\chi)(-j)).
 \end{align*}

\end{proof}

\vskip 3em 
 
\section{Comparison with Kato's Euler system}

\noindent Let $f$ be any cuspidal newform of weight $k\geq 2$ and level $\Gamma_0(p^n)\cap \Gamma_1(N)$, $V_f$ the (cohomological) $p$-adic Galois representation of $G_{\mathbb{Q}}$ attached to $f$. We assume throughout this section that $V_f|_{G_{\mathbb{Q}_p}}$ is absolutely irreducible.

We have exactly two cases when the above assumption can happen, namely, when $\pi_p^{\sm}(f)$ is supercuspidal (as in section 3 to 5) or a twist of unramified principal series (as in Section 6 and 7).

\vskip 1em

\begin{cor}\label{main cor}
Let $f$ be as above. We define $\z_M(f):=\z_M^{+}(f)-\z_M^{-}(f)\in \HH_{\Iw}^1\left(\mathbb{Q}_p,V_f^*\right)$. Then we have for any $0\leq j\leq k-2$, any locally constant character $\chi$ of $\mathbb{Z}_p^*$ with conductor $p^n$ \normalfont{($n\geq 0$)}, \begin{equation}\label{main equation}
    \exp^*\left(\int_{\mathbb{Z}_p^*}\chi(x)x^{-j}\cdot\z_M(f)\right)=\frac{1}{\tau(\chi)}\cdot \frac{\widetilde{\Lambda}_{(p)}(f, \chi  ,j+1)}{j!}\cdot \overline{f}_{\chi}\cdot t^j,
\end{equation}
where
$$\widetilde{\Lambda}_{(p)}(f,\chi, j+1)=\frac{\Gamma(j+1)}{(2\pi i)^{j+1}}\cdot\frac{L_{(p)}(f,\chi, j+1)}{\Omega_f^{\pm}}\in\mathbb{Q}(f,\mu_{p^n})\text{ , \quad }\overline{f}_{\chi}\cdot t^j\in\Fil^0\D_{\dR}(V_f^*(\chi)(-j))$$are defined as in equation (\ref{lambda}) and Theorem \ref{thm-4.3}.

\end{cor}

\begin{proof}
This follows immediately from Theorem \ref{thm-4.3} (for the supercuspidal case) and Theorem \ref{thm-4.3ps} (for the principal series case).
\end{proof}

\vskip 1em

Recall the following theorem by Kato:

\begin{theorem}[Kato, \cite{K}]\label{5.1} There exists a \emph{unique} element  $\z_{\Kato}(f)\in\HH^1_{\Iw}\big(\mQ,V_f^*\big)$ $($in the \emph{global} Iwasawa cohomology group$)$, which is obtained by global methods using Siegel unites on modular curves, such that for any $0\leq j\leq k-2$ and any locally constant character $\chi$ of $\mathbb{Z}_p^*\cong\Gamma_{\mathbb{Q}_p}$ with conductor $p^n$ \normalfont{($n\geq 0$)}, we have \begin{equation}\exp^*\left(\int_{\mathbb{Z}_p^{*}}\chi(x)x^{-j}\cdot\z_{\Kato}(f)\right)=\frac{1}{\tau(\chi)}\cdot\frac{\widetilde{\Lambda}_{(p)}(f,\chi, j+1)}{j!}\cdot \overline{f}_{\chi}\cdot t^j.\end{equation}\end{theorem}
\begin{flushright}
$\Box$
\end{flushright}

We still use $\z_{\Kato}(f)$ to denote its image in the \emph{local} Iwasawa cohomology group $\HH^1_{\Iw}(\mathbb{Q}_p,V_f^*)$. Combining Corollary \ref{main cor} and Theorem \ref{5.1}, we can prove the following:

\vskip 1em

\begin{theorem}\label{7.2} If $V_f|_{G_{\mathbb{Q}_p}}$ is absolutely irreducible, then $\z_M(f)=\z_{\Kato}(f)$ as elements in $\HH^1_{\Iw}\big(\mQ_p,V_f^*\big)$.
\end{theorem}


\begin{proof}


We define $\HH^1_{\Iw,e}(\mathbb{Q}_p,V_f^*)$ the collection of elements in $\HH^1_{\Iw}(\mathbb{Q}_p,V_f^*)$ whose projection to $H^1(\mathbb{Q}_p(\zeta_{p^n}),V_f^*)$ belongs to $H^1_e(\mathbb{Q}_p(\zeta_{p^n}),V_f^*)$ for all $n\geq 0$, and $\HH^1_{\Iw,g}(\mathbb{Q}_p,V_f^*)$ the collection of elements in $\HH^1_{\Iw}(\mathbb{Q}_p,V_f^*)$ whose projection to $H^1(\mathbb{Q}_p(\zeta_{p^n}),V_f^*)$ belongs to $H^1_g(\mathbb{Q}_p(\zeta_{p^n}),V_f^*)$ for all $n\geq 0$, where \begin{align*}
    H^1_e\left(\mathbb{Q}_p(\zeta_{p^n}),V\right)&=\ker\left(H^1\left(\mathbb{Q}_p(\zeta_{p^n}),V\right)\rightarrow H^1\left(\mathbb{Q}_p(\zeta_{p^n}),V\otimes_{\mathbb{Q}_p}\B_{\crys}^{\varphi=1}\right)\right)\\
    &=\ker\left(H^1\left(\mathbb{Q}_p(\zeta_{p^n}),V\right)\xrightarrow{\exp^*}\Fil^0\D_{\dR}\left(V\right)\right)\\
    &\subset H^1_g\left(\mathbb{Q}_p(\zeta_{p^n}),V\right):=\ker\left(H^1\left(\mathbb{Q}_p(\zeta_{p^n}),V\right)\rightarrow H^1\left(\mathbb{Q}_p(\zeta_{p^n}),V\otimes_{\mathbb{Q}_p}\B_{\dR}\right)\right)\\
    &\qquad\qquad\qquad\qquad\subset H^1\left(\mathbb{Q}_p(\zeta_{p^n}),V\right)
\end{align*}
for any continuous $p$-adic Galois representation $V$ of $G_{\mQ_p}$.

By Corollary \ref{main cor} and Theorem \ref{5.1}, we know that in both cases, we have for any $0\leq j\leq k-2$ and any finite order character $\phi$ of $\mathbb{Z}_p^*$, 
\begin{equation}\label{5.1.1}
    \exp^*\left(\int_{\mathbb{Z}_p^*}\phi(x)x^{-j}\left(\z_M(f)-\z_{\Kato}(f)\right)\right)=0,
\end{equation}

This means $$\z_M(f)-\z_{\Kato}(f)\in \HH^1_{\Iw,e}(\mathbb{Q}_p,V_f^*).$$

In the case when $V_f|_{G_{\mathbb{Q}_p}}$ is absolutely irreducible, Proposition II.3.1 of \cite{Ber1} says \begin{equation}
   \HH^1_{\Iw, e}(\mathbb{Q}_p, V_f^*)\subset \HH^1_{\Iw, g}(\mathbb{Q}_p, V_f^*)=0,
\end{equation} hence $\z_M(f)-\z_{\Kato}(f)=0$, which proves (1).

\end{proof}

\vskip 1em

\begin{remark}
In fact, we only used equation \eqref{5.1.1} for $j=0$ in order to conclude $\z_M(f)=\z_{\Kato}(f)$ as elements in $\HH^1_{\Iw}\big(\mQ_p,V_f^*\big)$. \end{remark}

\vskip 5em

\bibliographystyle{amsplain}

\end{document}